\documentclass[12pt]{amsart}
\usepackage{epsfig,color}
\usepackage[dvipsnames]{xcolor}
\numberwithin{figure}{section}
\numberwithin{table}{section}
\newtheorem{theorem}{Theorem}[section]

\newtheorem{lemma}[theorem]{Lemma}
\newtheorem{prop}[theorem]{Proposition}
\newtheorem{conjecture}[theorem]{Conjecture}
\newtheorem{cor}[theorem]{Corollary}
\theoremstyle{definition}
\newtheorem{definition}[theorem]{Definition}
\theoremstyle{remark}
\newtheorem{remark}[theorem]{Remark}

\def\G{\Gamma}
\def\t{\widetilde}

\def \sgn{\mathrm{sgn }}

\begin{document}

\title[Cluster algebras from surfaces and extended affine Weyl groups]{Cluster algebras from surfaces \\ and extended affine Weyl groups}
\author{Anna Felikson, John W. Lawson, Michael Shapiro, and Pavel Tumarkin}
\address{Department of Mathematical Sciences, Durham University, Science Laboratories, South Road, Durham, DH1 3LE, UK}
\email{anna.felikson@durham.ac.uk, pavel.tumarkin@durham.ac.uk}
\thanks{The research was partially supported by EPSRC grant EP/N005457/1 (A.F.), EPSRC PhD scholarship (J.W.L.), NSF grant DMS-1702115 and International Laboratory of Cluster Geometry NRU HSE, RF Government grant (M.S.), and Leverhulme Trust grant RPG-2019-153 (P.T.)}

\address{Codeplay Software Ltd, 3 Lady Lawson Street, Edinburgh, EH3 9DR, UK}
\email{john@codeplay.com}

\address{Department of Mathematics, Michigan State University, East Lansing, MI 48824, USA
  \and
National Research University Higher School of Economics, Russia}
\email{mshapiro@math.msu.edu}

\begin{abstract}
  We characterize mutation-finite cluster algebras of rank at least 3 using positive semi-definite quadratic forms. In particular, we associate with every unpunctured bordered surface a positive semi-definite quadratic space $V$, and with every triangulation a basis in $V$, such that any mutation of a cluster (i.e., a flip of a triangulation) transforms the corresponding bases into each other by partial reflections. Furthermore, every triangulation gives rise to an extended affine Weyl group of type $A$, which is invariant under flips. The construction is also extended to exceptional skew-symmetric mutation-finite cluster algebras of types $E$.
  
\end{abstract}

\dedicatory{To the memory of Ernest Borisovich Vinberg}

\maketitle
\setcounter{tocdepth}{1}
\tableofcontents

\section{Introduction and main results}

Cluster algebras were introduced by Fomin and Zelevinsky~\cite{FZ} in an effort to describe dual canonical bases in the universal enveloping algebras of Borel subalgebras of simple complex Lie algebras. {A cluster algebra possesses the distinguished set of generators, \emph{cluster variables}, organized in the groups of the same cardinality called \emph{clusters} which form a combinatorial structure described by an \emph{exchange graph}, where clusters correspond to vertices of the exchange graph. The generators of the neighboring clusters are algebraically dependent, where the corresponding relations are encoded by \emph{exchange matrix} or, equivalently, by an \emph{exchange quiver}.} In the famous paper~\cite{FZ2} {  Fomin and Zelevinsky} obtained Cartan-Killing type classification of all cluster algebras of finite type, i.e. cluster algebras having only finitely many distinct cluster variables. A wider class of cluster algebras is formed by cluster algebras {\em of finite mutation type} which have finitely many exchange matrices (but are allowed to have infinitely many cluster variables). These algebras found various applications, including ones in quantum field theories (see e.g.~\cite{CV,CV2}).

Skew-symmetric cluster algebras {  (i.e., with skew-symmetric exchange matrices)} of finite mutation type were classified in~\cite{FeSTu}: it was shown there that such algebra either has rank at most two, or corresponds to a triangulated surface, or belongs to one of finitely many exceptional mutation classes. 
The approach in~\cite{FeSTu} was based on a computer assisted analysis of the combinatorial  structure of the exchange graph of mutation-finite cluster algebras. This { paper} is written in an effort to develop a more conceptual characterization of the mutation finite phenomenon. { We consider the property of mutation finiteness to be related to the existence} of some positive semi-definite symmetric form similar to the classification of finite type cluster algebras \cite{FZ2} and to the classification of reflection groups of finite and affine types~\cite{Co}.

In the present paper, we follow the path started in \cite{BGZ,Se1,Se,S2}  characterizing mutation-finite cluster algebras of rank at least 3 using associated quadratic forms called {\em quasi-Cartan companions}.

The notion of a quasi-Cartan companion of a skew-symmetric matrix $B$ (or, equivalently, of the corresponding quiver) was introduced in~\cite{BGZ} as a symmetric matrix whose off-diagonal entries have the same moduli as ones of $B$ (we recall  the precise definitions in Section~\ref{semi-def}). It was proved in~\cite{BGZ} that a matrix $B$ defines a cluster algebra of finite type if and only if it has a positive definite quasi-Cartan companion, and all the cycles in the associated to $B$ quiver are cyclically oriented.    

This result was extended to the case of algebras of affine type in~\cite{Se}, where it was proved that a matrix defines a cluster algebra of affine type if and only if it has a positive semi-definite quasi-Cartan companion of corank one satisfying some additional {\em admissibility conditions} (see Section~\ref{adm-sec}).

Our first construction provides the following result (for brevity, we formulate everything in terms of quivers, see Section~\ref{background} for details).

\setcounter{section}{3}
\setcounter{theorem}{13}
\begin{theorem}
Let $Q$ be a connected quiver of finite mutation type with at least $3$ vertices.
Then $Q$ has a positive semi-definite quasi-Cartan companion.
\end{theorem}

As a corollary, we obtain the following characterization of finite mutation classes.

\begin{cor}
A connected quiver $Q$ with at least three vertices is mutation-finite if and only if every quiver in the mutation class of $Q$ has a positive semi-definite quasi-Cartan companion. 
\end{cor}

A quasi-Cartan companion $A$ of a quiver $Q$ can be mutated along with the quiver (we recall the definition given in~\cite{BGZ} and our geometric interpretation of it in Section~\ref{semi-def}). Understanding $A$ as a matrix of a quadratic form on a real vector space, and thus as a Gram matrix of a certain basis (called a {\em companion basis}~\cite{P1,P2}), the mutation corresponds to a change of basis (we call this procedure a {\em mutation of a basis}). However, the mutated matrix $\mu_k(A)$ may not be a quasi-Cartan companion of the mutated quiver $\mu_k(Q)$. A notion of $k$-{\em compatible} companion was introduced in~\cite{BGZ} to guarantee that $\mu_k(A)$ is again a quasi-Cartan companion of $\mu_k(Q)$. We introduce a notion of a {\em fully compatible} companion which is $k$-compatible for every vertex $k$, and thus its mutation in every direction leads to a quasi-Cartan companion of the mutated quiver. We then prove the following result.  

\setcounter{theorem}{16}
\begin{theorem}
Let $Q$ be a mutation-finite quiver with at least $3$ vertices. If $Q$ is not
the quiver shown in Fig.~\ref{killhope} then $Q$ has a fully compatible positive semi-definite quasi-Cartan companion.
\end{theorem}

Although every mutation of a fully compatible positive semi-definite quasi-Cartan companion provided by Theorem~\ref{thm fully} is again a positive semi-definite quasi-Cartan companion, it may not be fully compatible, and thus further mutation may not lead to a quasi-Cartan companion. So, the main question we want to explore is when we are able to mutate a quasi-Cartan companion throughout the whole mutation class of a quiver {  (we call such a quasi-Cartan companion a {\em symmetric twin} of a quiver)}. This is the case, for example, for quivers without oriented cycles: it was proved in~\cite{ST} that a quasi-Cartan companion of such a quiver with all off-diagonal entries being non-positive can be mutated along any mutation sequence to produce a quasi-Cartan companion again.

Our main result is the following. For every unpunctured surface we pick a specific representative from its mutation class of quivers, and then construct a companion basis ${\bf u}$ belonging to an extended affine root system of type $A_{n-n_0}^{[n_0]}$ (we recall the definitions in Section~\ref{eawg}), which agrees with the results of~\cite{CdZ} (where $n$ is the rank of the quiver, and $n_0$ is the dimension of the kernel of the corresponding quadratic form which can be expressed in terms of Euler characteristics of $S$). Then the next theorem states that this collection of roots gives rise to a positive semi-definite {  symmetric twin} for every quiver in the mutation class.     

\setcounter{section}{5}
\setcounter{theorem}{4}
\begin{theorem}
Given a quiver $Q$ constructed from a triangulation of an unpunctured surface, there exists a companion basis ${\bf u}$ for $Q$ such that for any mutation sequence $\mu$ the Gram matrix of vectors $\mu(\bf u)$ is a positive semi-definite quasi-Cartan companion for $\mu(Q)$.   

\end{theorem}

{   
Moreover, we claim in Corollary~\ref{all-un} that the choice of such basis ${\bf u}$ is essentially unique up to a linear isometry and sign changes of vectors in ${\bf u}$.
}

We approach Theorem~\ref{c-ind} in two different ways. The first one is through the groups constructed in~\cite{FeTu} by quivers originating from unpunctured surfaces. With any such quiver we can then associate two groups, whose generators are involutions indexed by vertices of the quiver: a group $G$ constructed in~\cite{FeTu} (which is a quotient of a certain Coxeter group), and an extended affine Weyl group $W$ generated by reflections. These groups turned out to be related in the following way.       

\setcounter{section}{5}
\setcounter{theorem}{0}
\begin{theorem}
There exists a surjective homomorphism  $\varphi: G\to W$ taking generating involutions of $G$ to generating reflections of $W$.
\end{theorem}

In finite and affine types the groups $G$ and $W$ are actually isomorphic, which gives rise to the following conjecture.

\begin{conjecture}
The map $\varphi$ in Theorem~\ref{homo} is an isomorphism.
  \end{conjecture} 

  We note that both groups $G$ and $W$ can also be defined for all $9$ exceptional mutation-finite classes of types $E$ (see Fig.~\ref{exceptional-fig}), and the conjecture holds in these cases (see Remark~\ref{iso-el}).

Assuming Conjecture~\ref{iso}, Theorem~\ref{c-ind} holds easily (see Prop.~\ref{ind}).

As we do not have a proof of Conjecture~\ref{iso}, we prove Theorem~\ref{c-ind} in a different way by using the notion of an {\em admissible} quasi-Cartan companion introduced in~\cite{Se}. The admissibility condition is stronger than full compatibility, so Theorem~\ref{c-ind} can be deduced from the following result which we prove in Section~\ref{adm-sec}.
  
\setcounter{section}{6}
\setcounter{theorem}{1}
\begin{prop}\footnote{While preparing the paper we were informed by Ahmet Seven that he has obtained an independent proof of Proposition~\ref{adm}.}
Given a quiver $Q$ constructed from a triangulation of an unpunctured surface, there exists an admissible quasi-Cartan companion $A$ of $Q$ such that for any mutation sequence $\mu$ the matrix $\mu(A)$ is an admissible quasi-Cartan companion of $\mu(Q)$. { In particular, $A$ is a symmetric twin of $Q$.} 
\end{prop}

\bigskip

The paper is organized as follows. Section~\ref{background} contains the essential facts about mutations of quivers, and about finite mutation classes. In Section~\ref{semi-def}, we first recall the basics on quasi-Cartan companions, and then construct positive semi-definite quasi-Cartan companions for all mutation-finite quivers. In Section~\ref{group},  we discuss the group constructed from a quiver in~\cite{FeTu}: after recalling the presentation, we show that the group depends on three numerical parameters only, namely, on the topological type of the surface (genus and the number of boundary components), and the number of marked points. In other words, the group turns out to be independent on the distribution of marked points amongst the boundary components (note that if Conjecture~\ref{iso} is true, the group would depend on two numerical parameters only). In Section~\ref{eawg}, we associate  with every unpunctured surface an extended affine Weyl group of type $A$, and then discuss the relations between this group and the one constructed above. Finally, Section~\ref{adm-sec} is devoted to the proof of Proposition~\ref{adm} and Theorem~\ref{c-ind} by using admissible quasi-Cartan companions. We also discuss the geometric interpretation of the admissibility condition.

\subsection*{Acknowledgements} The authors are grateful to the anonymous referee for valuable comments. The authors would like to express their gratitude to the Research Institute for Mathematical Sciences, Kyoto, and the organizers of the program on Cluster Algebras at RIMS in the Spring of 2019. M.S. is also grateful to the Research in Pairs Program at the Mathematisches Forschungsinstitut Oberwolfach (Summer 2019) and Mathematical Science Research Institute, Berkeley (Fall 2019) for their hospitality and outstanding working conditions they provided.

  \setcounter{section}{1}
  
\section{Quivers of finite mutation type}
\label{background}
 In this section, we recall the essential notions on mutations of quivers of finite, affine, and finite mutation type.
For details see~\cite{FST}.

\subsection{Quivers and mutations}
\label{dm}
An $n\times n$ skew-symmetric integer matrix $B$ can be encoded by a {\em quiver} $Q$ which is a (multi)-graph with oriented edges (called {\it arrows}). 
Vertices of $Q$ are labeled by $[1,\dots,n]$. If $b_{ij}>0$, we join vertices $i$ and $j$ by $b_{ij}$ arrows directed from $i$ to $j$. 
Throughout the paper we assume that all diagrams are connected (equivalently, matrix $B$ is assumed to be indecomposable). 

For every vertex $k$ of a quiver $Q$ one can define an involutive operation  $\mu_k$ called {\it mutation of $Q$ in direction $k$}. This operation produces a new quiver  denoted by $\mu_k(Q)$ which can be obtained from $Q$ in the following way (see~\cite{FZ}): 
\begin{itemize}
\item
orientations of all arrows incident to a vertex $k$ are reversed; 
\item
for every pair of vertices $(i,j)$ such that $Q$ contains arrows directed from $i$ to $k$ and from $k$ to $j$ the number of arrows joining $i$ and $j$ changes as described in Figure~\ref{quivermut}.
\end{itemize} 

\begin{figure}[!h]
\begin{center}
\epsfig{file=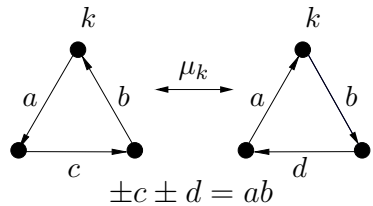,width=0.3\linewidth}
\put(-132,20){\small $a$}
\put(-96,20){\small $b$}
\put(-46,20){\small $a$}
\put(-10,20){\small $b$}
\put(-115,-7){\small $c$}
\put(-30,-7){\small $d$}
\put(-110,50){\small $k$}
\put(-25,50){\small $k$}
\put(-73,32){\small $\mu_k$}\\
$\pm{c}\pm{d}={ab}$
\caption{Mutations of quivers. The sign before ${c}$ (resp., ${d}$) is positive if the three vertices form an oriented cycle, and negative otherwise. Either $c$ or $d$ may vanish. If $ab$ is equal to zero then neither value of $c$ nor orientation of the corresponding arrow does change.}
\label{quivermut}

\end{center}
\end{figure}

Given a quiver $Q$, its {\it mutation class} is the set of all quivers obtained from
the given one by all sequences of iterated mutations. All quivers from one mutation class are called {\it mutation-equivalent}.

\subsection{Finite type}

A quiver is of {\it finite type} if it is mutation-equivalent to an orientation of a simply-laced Dynkin diagram.
So, a quiver of finite type is of one of the following mutation types: 
$A_n$, $D_n$, $E_6$, $E_7$ or $E_8$.

It is shown in~\cite{FZ2} that mutation classes of quivers of finite type are in one-to-one correspondence
with skew-symmetric cluster algebras of finite type. In particular, this implies that any subquiver of a quiver of finite type is also of finite type.

\subsection{Affine type}
A quiver is of {\it affine type} if it is mutation-equivalent to an orientation of a simply-laced affine Dynkin diagram different from an oriented cycle.
A quiver of affine type is of one of the following mutation types: 
$\widetilde A_{k,n-k}$, $0<k<n$ (see Remark~\ref{a}), $\widetilde D_n$, $\widetilde E_6$, $\widetilde E_7$ or $\widetilde E_8$.

\begin{remark}
\label{a}
Let $\widetilde D$ be an affine Dynkin diagram different from $\widetilde A_n$. Then all orientations of $\widetilde D$ are mutation-equivalent. The orientations of  $\widetilde A_{n-1}$ split into $[n/2]$ mutation classes $\widetilde A_{k,n-k}$ (where by $[x]$ we mean the integer part of $x$): each class contains a cyclic representative with only two changes of orientations, with $k$ consecutive arrows in one direction and $n-k$ in the other, $0<k<n$.

\end{remark} 

We will heavily use the following statement. 
\begin{prop}[\cite{BMR,Z}]
\label{subd of aff}
Any subquiver of a quiver of affine type is either of finite or of affine type.

\end{prop}

\subsection{Finite mutation type }

A quiver  is called {\it mutation-finite} (or {\it of finite mutation type}) if its mutation class is finite.

As it is shown in~\cite{FeSTu}, a quiver of finite mutation type
either has only two vertices, or corresponds to a triangulated surface (see Section~\ref{triang-sec}), or belongs to one of finitely many exceptional mutation classes.

\begin{theorem}[\cite{FeSTu}]
\label{class}
Let $\G$ be a mutation-finite diagram with at least $3$ vertices. Then either $\G$ arises from a triangulated surface, or $\G$ is mutation-equivalent to one of $18$ exceptional diagrams
$E_6,E_7,E_8, \widetilde E_6,\widetilde E_7,\widetilde E_8,E_6^{(1,1)}\!,E_7^{(1,1)}\!,E_8^{(1,1)}\!,X_6,X_7$
shown in Fig.~\ref{exceptional-fig}.

\end{theorem}

\begin{figure}[!h]
\begin{center}
\epsfig{file=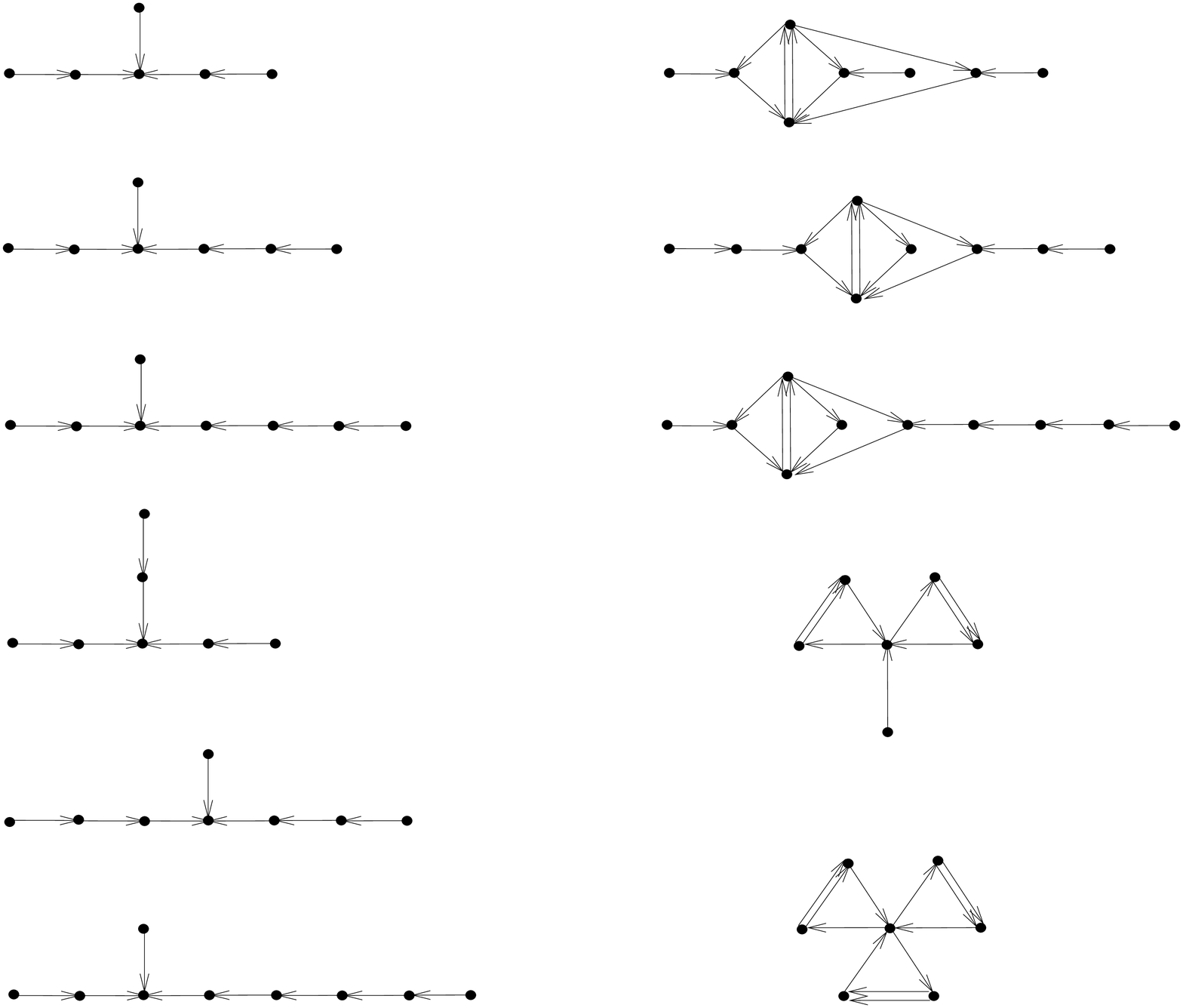,width=0.89\linewidth}
\put(-420,305){$E_6$}
\put(-420,245){$E_7$}
\put(-420,185){$E_8$}
\put(-420,110){$\t E_6$}
\put(-420,48){$\t E_7$}
\put(-420,-10){$\t E_8$}
\put(-210,305){$E_6^{(1,1)}$}
\put(-210,245){$E_7^{(1,1)}$}
\put(-210,185){$E_8^{(1,1)}$}
\put(-160,110){$X_6$}
\put(-160,10){$X_7$}
\end{center}
\caption{Exceptional finite mutation classes}
\label{exceptional-fig}
\end{figure}

\subsection{Triangulated surfaces and block-decomposable quivers}
\label{triang-sec}

The correspondence between quivers of finite mutation type and triangulated surfaces  is developed in~\cite{FST}. Here we briefly remind the basic definitions.

By a {\it surface} we mean a genus $g$ orientable surface with $r$ boundary components and a finite set of marked points, with at least one marked point at each boundary component. A non-boundary marked point is called a {\it puncture}. 

An (ideal) {\it triangulation} of a surface is a triangulation with vertices of triangles in the marked points. We allow self-folded triangles and  follow~\cite{FST} considering triangulations as {\it tagged triangulations} (however, we are neither reproducing nor using all the details in this paper).

Given a triangulated surface, one constructs a quiver in the following way:
\begin{itemize}
\item vertices of the quiver correspond to the (non-boundary) edges of a triangulation;
\item two vertices are connected by an arrow if they correspond to two sides of the same triangle (i.e., there is one simple arrow between given two vertices for every such triangle);
inside the triangle orientations of the arrow are arranged counter-clockwise (with respect to some orientation of the surface);
\item two arrows with different directions connecting the same vertices cancel out;
two arrows in the same direction result in a double arrow;
\item for a self-folded triangle (with two sides identified), two vertices of the quiver corresponding to the sides of this triangle are disjoint;
a vertex corresponding to the ``inner'' side of the triangle is connected to other vertices in the same way as the vertex corresponding to the outer side of the triangle.

\end{itemize}

It is shown in~\cite{FST} that any surface can be cut into {\it elementary surfaces}, we list their quivers in Fig.~\ref{blocks-list}.  
We use {\it white} color for the vertices corresponding to the ``exterior'' edges of these elementary surfaces (such vertices are called {\em open}) and {\it black} for the vertices corresponding to ``interior'' edges. The quivers in  Fig.~\ref{blocks-list} are called {\it blocks}.
Depending on a block, we call it {\it a block of type} ${\rm{I}}$, ${\rm{II}}$ etc. 

As elementary surfaces are glued to each other to form a triangulated surface, the blocks are glued to form a {\it block-decomposition} of a bigger quiver. 
A connected quiver $Q$ is called {\it block-decomposable} (or simply, {\it decomposable})
if it can be obtained from a collection of blocks by identifying white (i.e., open) vertices of different blocks along some partial matching (matching of vertices of the same block is not allowed), where two simple arrows with the same endpoints and opposite directions cancel out, and two arrows with the same endpoints and the same directions form a double arrow. A non-connected quiver $Q$ is called  block-decomposable if every connected component of $Q$ is either decomposable or a single vertex. 

\begin{figure}[!h]
\begin{center}
\epsfig{file=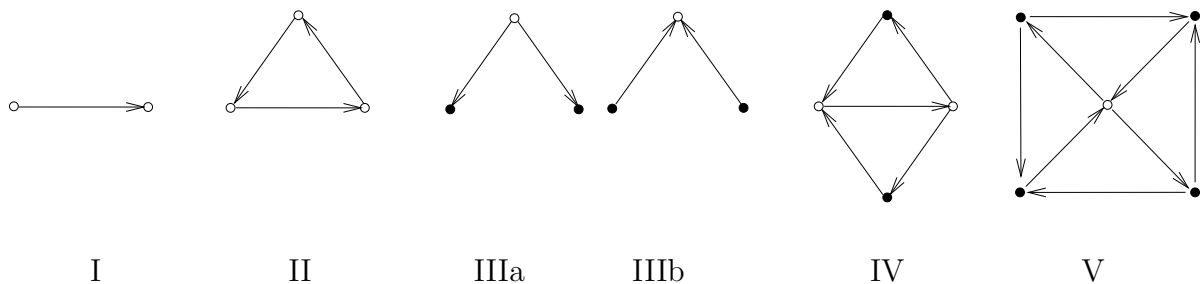,width=0.99\linewidth}
\put(-420,-30){I}
\put(-345,-30){II}
\put(-275,-30){IIIa}
\put(-215,-30){IIIb}
\put(-125,-30){IV}
\put(-45,-30){V}
\caption{Blocks used to obtain quivers from triangulations}
\label{blocks-list}
\end{center}
\end{figure}

Block-decomposable quivers are in one-to-one correspondence with adjacency matrices of arcs of ideal (tagged) triangulations of bordered two-dimensional surfaces with marked points (see~\cite[Section~13]{FST} for the detailed explanations). Mutations of block-decomposable quivers correspond to flips of (tagged) triangulations. In particular, this implies that mutation class of any block-decomposable quiver is finite, and any subquiver of a block-decomposable diagram is block-decomposable too.

Theorem~\ref{class} shows that block-decomposable quivers almost exhaust mutation-finite ones.

We will use the surface presentations of block-decomposable quivers of finite and affine type, see Table~\ref{surface realizations}.

\begin{table}[!h]
\begin{center}
\caption{Surfaces corresponding to quivers of finite and affine type}
\label{surface realizations}
\begin{tabular}{|c|c|c|}
\hline
\multicolumn{3}{|c|}{Finite types}\\
\hline
  \raisebox{0mm}{$A_n$, $n\ge 1$} &\raisebox{-1mm}{\epsfig{file=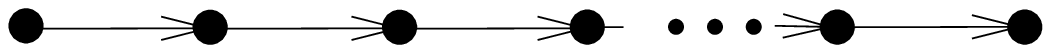,width=0.25\linewidth}} &\raisebox{0mm}{\small disk}\\
\hline
\raisebox{5mm}{$D_n$, $n\ge 4$}&\epsfig{file=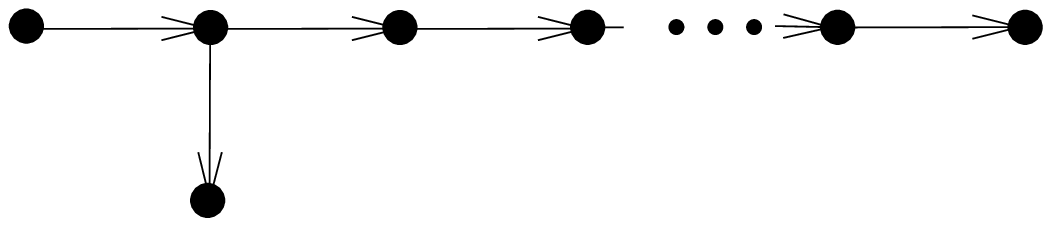,width=0.25\linewidth} &\raisebox{4mm}{\small punctured disk}\\
                                                  \hline
  \multicolumn{3}{|c|}{Affine types}\\
  \hline
\raisebox{4mm}{$\widetilde A_{k,n-k}$, $n>k\ge 1$} & \raisebox{0mm}{\epsfig{file=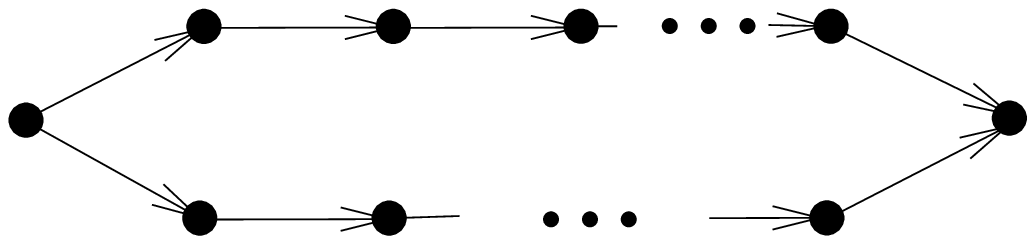,width=0.25\linewidth}}&\raisebox{4mm}{\small annulus}\\
\hline
\raisebox{5mm}{$\widetilde D_n$, $n\ge 4$}& \epsfig{file=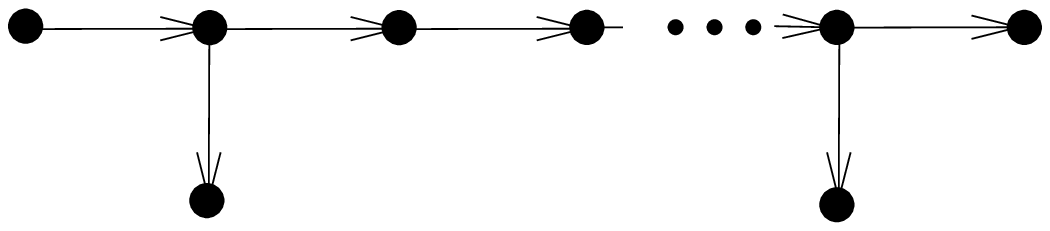,width=0.25\linewidth}&\raisebox{5mm}{\small twice punctured disk}\\
\hline
\end{tabular}
\end{center}
\end{table}

\begin{remark}
A mutation class $\widetilde A_{k,n-k}$ (of affine type $\widetilde A_{n-1}$) corresponds to an annulus with $k$ marked points on one boundary component and $n-k$ on the other.

\end{remark}

\subsection{Subquivers of mutation-finite quivers}
\label{subd}

In this section, we list some technical facts we are going to use in the sequel. 

\subsubsection*{Oriented cycles in mutation-finite quivers}

It is easy to see that there are two types of mutation-finite oriented chordless cycles: simply-laced cycles (they are of finite type $D_n$) and a cycle of length three with $(1,1,2)$ arrows (of type $\t A_{2,1}$). Note that a quiver of type $D_n$ for $n\ge 4$ corresponds to a punctured disk, so these will not appear in quivers constructed from  unpunctured surfaces.    

\subsubsection*{Non-oriented cycles in mutation-finite quivers}

It is also a well known fact that all non-oriented cycles in mutation-finite quivers are simply-laced (and thus of type $\t A_{k,n-k}$ for some $k$).

We will also use the following statement proved by Seven in~\cite{Se1}.

\begin{prop}[Proposition~2.1(iv),~\cite{Se1}]
\label{non-or}
Let $Q$ be a simply-laced mutation-finite quiver and let $C\subset Q$ be a non-oriented chordless cycle.
Then for each vertex $v\in Q$ the number of arrows connecting $v$ with $C$ is even.

\end{prop}

\section{Quasi-Cartan companions}
\label{semi-def}

In~\cite{BGZ}, Barot, Geiss and Zelevinsky introduced a notion of quasi-Cartan companion of a skew-symmetrizable matrix and defined its mutation. As we restrict ourselves to quivers, we reproduce below their definitions for skew-symmetric matrices. 

\subsection{Definitions and basic properties}

\begin{definition}[Quasi-Cartan companion]
Let $B$ be an $n\times n$  skew-symmetric matrix. An $n\times n$  symmetric matrix $A$ is a {\it quasi-Cartan companion} of $B$
if $|a_{ij}|=|b_{ij}|$ and $a_{ii}=2$.  
  
\end{definition}  

\begin{remark}
  A quasi-Cartan companion contains the same information as the skew-symmetric matrix $B$ together with the choice of signs assigned to each (unordered) pair of indices $(i,j)$, $1\le i,j \le n$ with non-zero $b_{ij}$ (sign of the entry $a_{ij}=a_{ji}$ in $A$).
  
Pictorially, we will represent a quasi-Cartan companion by labelling the arrows of a quiver with the signs of the corresponding elements.  

We will also say that a quasi-Cartan companion of a skew-symmetric matrix is a quasi-Cartan companion of the corresponding quiver.
\end{remark}  

\begin{definition}
Given a quiver $Q$ and its quasi-Cartan companion $A$, consider a quadratic vector space $V$ defined by the quadratic form $A$. Let ${\bf v}= \{v_1,\dots,v_n\}$ be basis vectors in $V$ for which $A$ serves as the Gram matrix, i.e. $(v_i,v_j)=a_{ij}$. Generalizing the definition of Parsons~\cite{P1,P2}, we will call the set of vectors $\bf v$ a {\it companion basis of $Q$}.

\end{definition}

\begin{definition}[Mutation of  quasi-Cartan companions]
Let $A$ be a quasi-Cartan companion of $B$.
A mutation $\mu_k$ of $A$ is defined as $\mu_k(A)=A'$ where
$$
a_{ij}'=\begin{cases}
2 & \text{if $i=j$; }\\
\sgn(b_{ik})a_{ik}  & \text{if $j=k$; }\\  
-\sgn(b_{kj})a_{kj}  & \text{if $i=k$; }\\  
a_{ij}-\sgn(a_{ik}a_{kj})[b_{ik}b_{kj}]_+  & \text{otherwise. }\\  
\end{cases}
$$
  
\end{definition}




\begin{remark} 
\label{geometric realisation}
There is a geometric interpretation of the mutation of quasi-Cartan companions, as follows.
Let $\bf v$ be a companion basis for $Q$. 
Then it is straightforward to check that the elements $a_{ij}'$ of $\mu_k(A)=A'$ satisfy $a_{ij}'=\langle v_i',v_j'\rangle$, where
%
$$
v_i'=\begin{cases}
-v_i &  \text{if $i=k$; }\\
v_i-\langle v_i,v_k\rangle v_k     &  \text{if $b_{i,k}>0$; }\\
v_i  &  \text{otherwise}.\\
\end{cases}  
$$
In other words, a mutation of a quasi-Cartan companion corresponds to the reflection of some of the vectors of the companion basis. In particular, $\mu_k(A)$ and $A$ define the same quadratic form (written in different bases).

\end{remark}

Note that the result of a mutation of a quasi-Cartan companion is not always a quasi-Cartan companion of the mutated matrix, see Fig.~\ref{ex-mut}.

\begin{figure}[!h]
\begin{center}
\epsfig{file=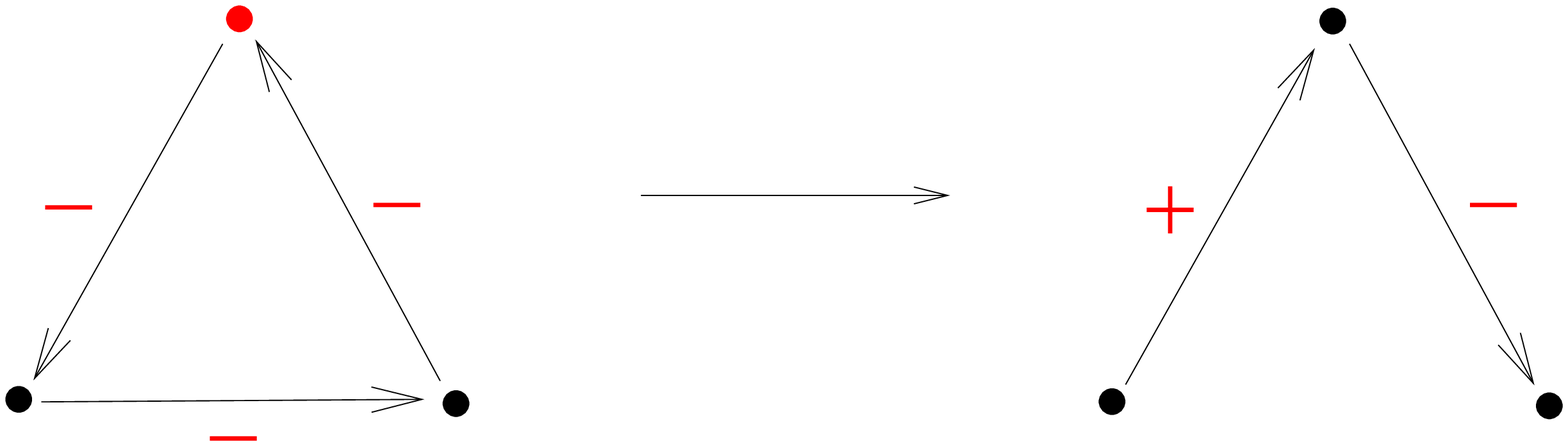,width=0.4\linewidth}
\put(-100,31){$\mu_2$}
\put(-210,-12){$e_1-e_2$}    
\put(-140,-12){$e_2-e_3$}    
\put(-150,45){$e_3-e_1$}    
\put(-80,-12){$e_1-e_2$}    
\put(-10,-12){$e_2-e_1$}    
\put(-20,45){$e_1-e_3$}    
\put(30,20){\scriptsize $\mu_2(A)=\begin{pmatrix}2& 1 & -2 \\ 1& 2 & -1 \\ -2 & -1 & 2 \end{pmatrix}   $}
\put(-290,20){\scriptsize $A=\begin{pmatrix}2& -1 & -1 \\ -1& 2 & -1 \\ -1 & -1 & 2 \end{pmatrix}   $}
\caption{A mutation transforming a quasi-Cartan companion of the quiver to a matrix which is not a quasi-Cartan companion of the mutated quiver.}
\label{ex-mut}
\end{center}
\end{figure}

In~\cite{BGZ},  Barot, Geiss and Zelevinsky described a sufficient condition for a quasi-Cartan companion to ensure that the result of the mutation $\mu_k$  is a quasi-Cartan companion of the mutated matrix. This is provided by the notion of $k$-compatibility.

\begin{definition}[$k$-compatibility]
 A quasi-Cartan companion $A$ is {\it $k$-compatible } if for every $i,j\ne k$  one has
 $$
 \begin{cases}
 a_{ij}a_{jk}a_{ki}>0  & \text{if $(i,j,k)$ form an oriented cycle,}\\
 a_{ij}a_{jk}a_{ki}\le 0  & \text{otherwise}.\\
 \end{cases}  
 $$
\end{definition}  

\begin{lemma}[\cite{BGZ}]
\label{l k-comp}   
Let $A$ be a $k$-compatible quasi-Cartan companion for $B$. Then $\mu_k(A)$ is a $k$-compatible quasi-Cartan companion for $\mu_k(B)$.
  
\end{lemma}

\begin{definition}[full compatibility]
A quasi-Cartan companion is {\it fully compatible} if it is $k$-compatible for every $k\in \{1,\dots, n \}$.
  
\end{definition}

\begin{remark}
  \label{sign}

  Given a fully compatible quasi-Cartan companion, we can change the sign of any vector in a companion basis to obtain a new fully compatible quasi-Cartan companion. The resulting matrix differs by the signs of off-diagonal entries in a given row and column.

\end{remark}

To construct an example of a fully compatible quasi-Cartan companion, one can take any acyclic quiver and label all its arrows with the negative sign. 

\begin{theorem}[\cite{S2,ST}]
\label{ST}  
Let $Q$ be an acyclic quiver and $A$ be its  quasi-Cartan companion with $a_{ij}\le 0$ for all $i\ne j$. Then for every sequence $\mu$ of mutations the matrix $\mu(A)$ is a fully compatible quasi-Cartan companion of $\mu(Q)$. 

\end{theorem}  

\begin{remark}
\label{fin,aff}  
In particular, Theorem~\ref{ST} can be applied in all finite and affine cases.
If $B$ is of finite type, the quasi-Cartan companion described in Theorem~\ref{ST} is a positive definite quasi-Cartan matrix (cf. Remark~\ref{geometric realisation}).
Similarly, If $B$ is of affine type, then the corresponding quasi-Cartan companion is positive semi-definite.
\end{remark}

\begin{remark}
\label{ell}
One can extend the above construction of a positive semi-definite quasi-Cartan companion to the case of 
the elliptic quivers $E_6^{(1,1)}$,  $E_7^{(1,1)}$,  $E_8^{(1,1)}$ in the following way:
\begin{itemize}
\item[-] Let $Q$ be a quiver of one of the types above, consider a subquiver $Q'$ obtained by removing one of the ends of the double arrow. Quiver $Q'$ is of the type $\widetilde E_6$, $\widetilde E_7$ or $\widetilde E_8$ respectively. Now consider a quasi-Cartan companion $A'$ of $Q'$ (with all arrows  labelled by the negative sign).
\item[-] Let $x$ be the removed node in $Q$, and let $x'$ be the other end of the double arrow. Let $v_i$ be the vector assigned to $x'$ in a companion basis giving rise to $A'$. Assigning to $x$ a copy of $v_i$ we obtain a positive semi-definite quasi-Cartan companion $A$ of $Q$.

\end{itemize}

One can check explicitly by computation that for each mutation sequence $\mu$ the matrix $\mu(A)$  is a fully compatible quasi-Cartan companion of $\mu(Q)$ (this also follows fro Remark~\ref{iso-el} together with Proposition~\ref{adm}).
\end{remark}  

\begin{remark}
\label{basis}
Vectors constructed in Remark~\ref{ell} are clearly linearly dependent. However, we can slightly amend the construction by adding to $v_i$ a new basis vector lying in the kernel of the quadratic form. In this way we obtain a collection of linearly independent vectors which we have right to call a companion basis. We will follow this procedure throughout the paper. 
  \end{remark}



  \subsection{Positive semi-definite companions and mutation finiteness}
  In this section, we prove the following theorem.

\begin{theorem}
\label{thm fin mut type}
Let $Q$ be a connected quiver of finite mutation type with at least 3 vertices.
Then $Q$ has  a positive semi-definite quasi-Cartan companion.
  
\end{theorem}

\begin{proof}
As follows from the classification of quivers of finite mutation type (see~\cite{FeSTu}),
a quiver of finite mutation type is either of rank 2, or arises from a surface, or belongs to one of 11 exceptional mutation classes.   

For the exceptional quivers of finite and affine type the statement follows from Remark~\ref{fin,aff}.

For quivers of the types $E_6^{(1,1)}$,  $E_7^{(1,1)}$,  $E_8^{(1,1)}$ a positive semi-definite quasi-Cartan companion is constructed in Remark~\ref{ell}.

The mutation classes of quivers $X_6$ and $X_7$ are very small (containing $6$ and $2$ quivers respectively), for them one can check the statement directly. 

We are left to consider the case of quivers arising from triangulations of surfaces.
To build (a companion basis for) a quasi-Cartan companion of a quiver $Q$ originating from a given triangulation, we will assign vectors $v_1,\dots, v_n$ to the arcs of the triangulation, and the quasi-Cartan companion $A$ will be constructed as the Gram matrix of these vectors, i.e. $a_{ij}=(v_i,v_j)$. Let $t$ be the number of triangles in the triangulation. Consider a Euclidean  $t$-dimensional space with an orthonormal  basis $e_1,\dots,e_{t}$.
To construct the vectors $v_i$, we first assign the basis vectors $e_1,\dots,e_t$  to the triangles $T_1,\dots,T_t$ of the triangulation. To an arc  contained in the triangles $T_i$ and $T_j$ we assign a vector $e_i+e_j$ or $e_i-e_j$, as in Fig.~\ref{triang}. 
It is straightforward to see that the vectors constructed in this way provide a quasi-Cartan companion of the quiver $Q$. As they all lie in the Euclidean space, the quasi-Cartan companion is positive semi-definite. In view of Remark~\ref{basis}, we can also assume the vectors $v_1,\dots, v_n$ to be linearly independent.
  
\end{proof}  

\begin{figure}[!h]
\begin{center}
\epsfig{file=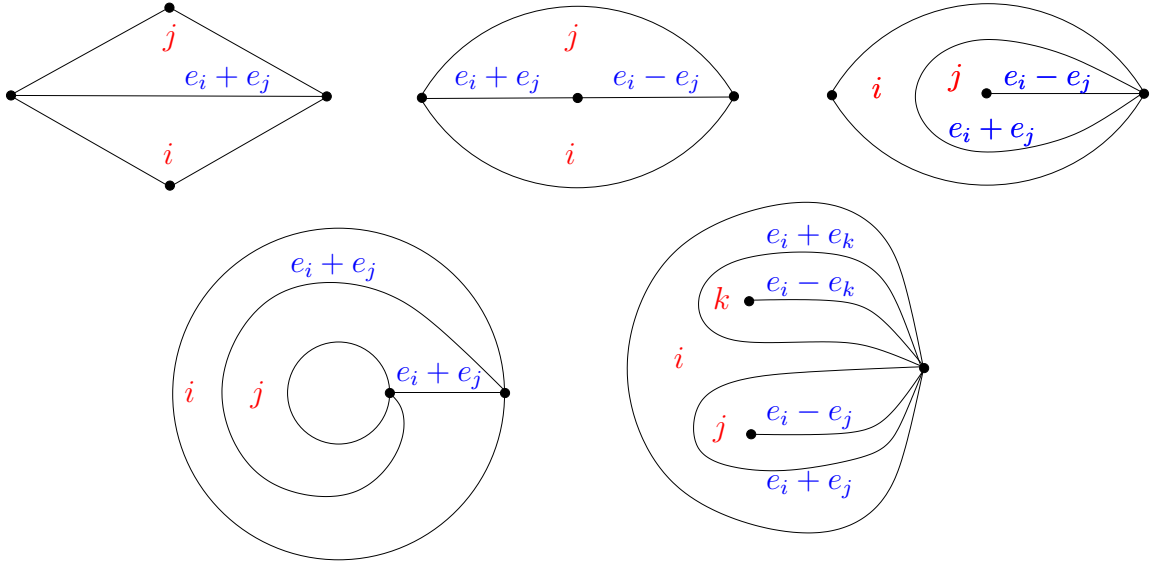,width=0.95\linewidth}
\put(-373,150){\color{red} $i$}
\put(-373,195){\color{red}$j$}
\put(-365,180){\color{blue} $e_i+e_j$}
\put(-221,150){\color{red} $i$}
\put(-221,195){\color{red} $j$}
\put(-203,180){\color{blue} $e_i-e_j$}
\put(-263,180){\color{blue} $e_i+e_j$}
\put(-105,175){\color{red} $i$}
\put(-76,180){\color{red} $j$}
\put(-55,180){\color{blue} $e_i-e_j$}
\put(-76,160){\color{blue} $e_i+e_j$}
\put(-105,175){\color{red} $i$}
\put(-76,180){\color{red} $j$}
\put(-55,180){\color{blue} $e_i-e_j$}
\put(-76,160){\color{blue} $e_i+e_j$}
\put(-365,60){\color{red} $i$}
\put(-340,60){\color{red} $j$}
\put(-325,109){\color{blue} $e_i+e_j$}
\put(-285,68){\color{blue} $e_i+e_j$}
\put(-180,73){\color{red} $i$}
\put(-165,47){\color{red} $j$}
\put(-165,95){\color{red} $k$}
\put(-145,53){\color{blue} $e_i-e_j$}
\put(-145,27){\color{blue} $e_i+e_j$}
\put(-145,102){\color{blue} $e_i-e_k$}
\put(-145,120){\color{blue} $e_i+e_k$}
\caption{Construction of vectors for quivers from triangulations. Triangles $i$ and $j$ are separated by an edge assigned with vector $e_i+e_j$, unless there are two common edges meeting at a puncture, see the configuration in the middle of the top row. An internal edge of self-folded triangle $j$ surrounded by triangle $i$ is assigned with vector $e_i-e_j$.}
\label{triang}
\end{center}
\end{figure}

\begin{cor}
  \label{char}
A connected quiver $Q$ of a rank higher than $2$ is mutation-finite if and only if every quiver in the mutation class of $Q$ has a positive semi-definite quasi-Cartan companion. 

\end{cor}

\begin{proof}
  In view of Theorem~\ref{thm fin mut type} it is sufficient to show that every mutation-infinite quiver $Q$ has a mutation-equivalent quiver not admitting a positive semi-definite quasi-Cartan companion. According to the well-known criterion (see e.g.~\cite[Corollary 8]{DO}), we can always find a quiver mutation-equivalent to $Q$ containing an arrow of weight at least $3$. Then any quasi-Cartan companion of such quiver is clearly indefinite.  

\end{proof}

One can ask a natural question whether there exists a {\it fully compatible } positive semi-definite quasi-Cartan companion for every mutation finite quiver.

\begin{remark}  
It is easy to check via case by case inspection (taking into account Remark~\ref{sign}) that
the quiver on Fig.~\ref{killhope} does not admit any fully compatible quasi-Cartan companion.

One can also check that this quiver corresponds to a closed torus with two punctures, and it  has one block decomposition only~\cite{Gu1,Gu2} (which consists of four blocks of type II). In particular, this quiver is not a subquiver of any larger quiver arising from a triangulation.

\end{remark}

\begin{figure}[!h]
\begin{center}
\epsfig{file=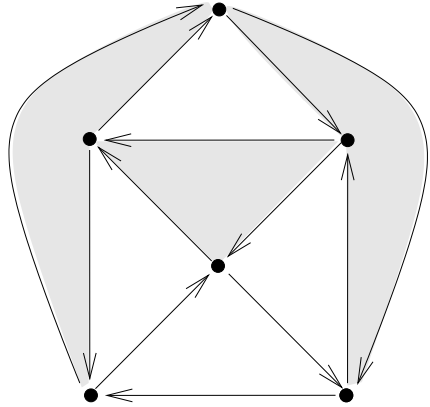,width=0.35\linewidth}
\caption{This quiver admits no fully compatible quasi-Cartan companion. Shaded triangles label non-oriented cycles.}
\label{killhope}
\end{center}
\end{figure}

In the next section we will show that this example is unique in the class of mutation-finite quivers.

\subsection{Fully compatible positive semi-definite companions}

The main result of this section is the following theorem.

\begin{theorem}
\label{thm fully}
Let $Q$ be a mutation-finite quiver with more than $2$ vertices. If $Q$ is not the quiver shown in Fig.~\ref{killhope} then $Q$ has a fully compatible positive semi-definite quasi-Cartan companion.

\end{theorem}

We need to show the statement for quivers originating from triangulations and for quivers from eleven exceptional mutation classes.

\begin{lemma}
\label{exceptional}
Theorem~\ref{thm fully} holds for quivers from  eleven exceptional mutation classes.

\end{lemma}

\begin{proof}
For finite, affine and elliptic quivers, the statement follows from Remarks~\ref{fin,aff} and~\ref{ell}.
For the quiver in mutation classes $X_6$ and $X_7$ the statement can be checked directly.

\end{proof}

Now, we are left to prove the theorem for the case of quivers from triangulations.

As in the proof of Theorem~\ref{thm fin mut type}, we will use vectors of the form $\pm e_i\pm e_j$, where   $e_1,\dots,e_n$ is an orthonormal basis of a Euclidean space (here vectors $\{e_i\}$ correspond to the triangles in the triangulation). In view of Remark~\ref{basis}, we assume the vectors are linearly independent.

We need the following technical definitions.

\begin{definition}
  If a vertex of $Q$ is assigned with a vector $v=\pm e_i \pm e_j$, we say that the set $\{e_i,e_j\}$ is the {\em support} of the vector $v$.
  
\end{definition}

\begin{definition}[Disjoint support companion basis]
Suppose that $Q$ is a quiver decomposed into blocks,  and suppose that vectors $\{v_k\}=\{ \pm e_{i_k} \pm e_{j_k}\}$ provide a companion basis for $Q$.
We say that $\{v_i\}$ is a  {\it  disjoint support companion basis} if for every open vertex $p_k$ of the block decomposition of $Q$ (i.e., an open vertex of some block which is not matched with any other) the following holds: if $p_k$ is not connected to any other open vertex in the block decomposition of $Q$ then the support of the vector  $v_k$ assigned to $p_k$ is not contained in the union of all supports of other vectors $\{v_i\}$ for $i\ne k$.


\end{definition}

We will use the following two technical lemmas concerning disjoint support companion bases.

\begin{lemma}
\label{bl_}
Every block has a fully compatible positive semi-definite quasi-Cartan companion with a disjoint support companion basis.

\end{lemma}

\begin{proof}
\label{bl}
The required companion bases are provided in Fig.~\ref{blocks}.

\end{proof}

\begin{figure}[!h]
\begin{center}
\epsfig{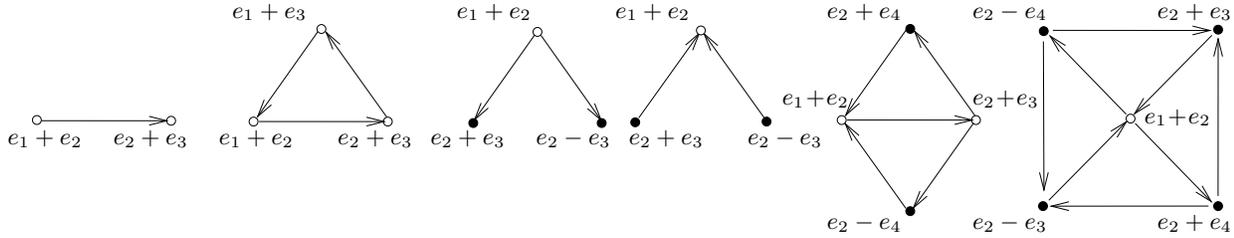}
\put(-460,27){\scriptsize $e_1+e_2$}
\put(-420,27){\scriptsize $e_2+e_3$}
\put(-380,27){\scriptsize $e_1+e_2$}
\put(-335,27){\scriptsize $e_2+e_3$}
\put(-375,75){\scriptsize $e_1+e_3$}
\put(-300,27){\scriptsize $e_2+e_3$}
\put(-260,27){\scriptsize $e_2-e_3$}
\put(-290,75){\scriptsize $e_1+e_2$}
\put(-225,27){\scriptsize $e_2+e_3$}
\put(-180,27){\scriptsize $e_2-e_3$}
\put(-230,75){\scriptsize $e_1+e_2$}
\put(-167,42){\scriptsize $e_1\!+\!e_2$}
\put(-95,42){\scriptsize $e_2\!+\!e_3$}
\put(-150,75){\scriptsize $e_2+e_4$}
\put(-150,-5){\scriptsize $e_2-e_4$}
\put(-95,75){\scriptsize $e_2-e_4$}
\put(-95,-5){\scriptsize $e_2-e_3$}
\put(-25,75){\scriptsize $e_2+e_3$}
\put(-25,-5){\scriptsize $e_2+e_4$}
\put(-30,35){\scriptsize $e_1\!+\!e_2$}
\caption{Fully compatible quasi-Cartan companions for blocks}
\label{blocks}
\end{center}
\end{figure}

\begin{lemma}
\label{disjoint}
Let $Q$ be a quiver decomposed into blocks. Suppose that $Q_1$ and $Q_2$ are subquivers of $Q$ such that $Q= Q_1\cup Q_2$, every block of the decomposition of $Q$ lies entirely either in $Q_1$ or in $Q_2$, and the intersection $Q_1\cap Q_2$ contains no arrows.
Suppose also that $Q_1$ and $Q_2$ have quasi-Cartan companions with  disjoint support companion bases.
Then $Q$ also has a quasi-Cartan companion with a disjoint support companion basis.

\end{lemma}

\begin{proof}
  Suppose that $\{p_1,\dots, p_k\}=Q_1\cap Q_2$ are vertices in the intersection. Consider disjoint support companion bases ${\bf w}=\{\pm w_i\pm w_j\}$ and ${\bf u}=\{\pm u_i\pm u_j\}$ of $Q_1$ and $Q_2$. Let $W$ and $U$ be vector spaces spanned by vectors $\{w_i\}$ and $\{u_i\}$ respectively. Denote by $w_{i_m}$ ($u_{i_m}$ resp.) the vectors showing up in the expressions assigned to vertices $p_m$, $m=1,\dots,k$. Now take the vector space $U+W$ by identifying  $w_{i_m}=u_{i_m}$, and extend the quadratic forms from $U$ and $W$ to $U+W$ by $(w_i,u_j)=0$ for all the remaining $i,j$. This provides a required  disjoint support companion basis.

\begin{figure}[!h]
\begin{center}
\epsfig{file=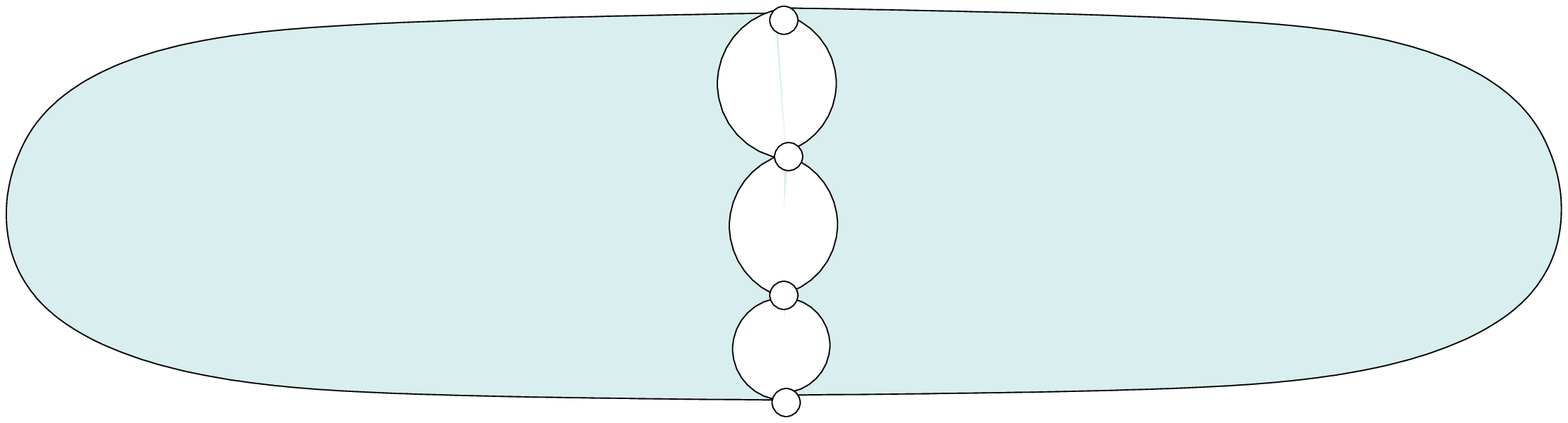,width=0.5\linewidth}
\put(-170, 40){$Q_1$}
\put(-80, 40){$Q_2$}
\put(-190, 20){\small $\pm w_i\pm w_j$}
\put(-80, 20){\small $\pm u_i\pm u_j$}
\caption{To the proof of Lemma~\ref{disjoint}.}
\label{disj}
\end{center}
\end{figure}

\end{proof}

To prove Theorem~\ref{thm fully}, we show first the statement for quivers with a block decomposition containing blocks of type II only. More precisely, for such quivers we will prove existence of a quasi-Cartan companion  with all required properties and additionally a disjoint support companion basis. This will be done by induction on the number of non-oriented triangles in the quiver (see Lemma~\ref{no_nonor} for the base of induction, i.e. the case all triangles are oriented, and Lemma~\ref{step} for the induction step).
We then show that one can also include blocks of type I and IV, and then finally we use Lemmas~\ref{bl_} and~\ref{disjoint} to conclude the theorem for block decompositions containing blocks of remaining types IIIa, IIIb and V.

\begin{lemma}
\label{no_nonor}
Let $Q$ be a quiver decomposed into several copies of block II. Suppose that $Q$ contains no non-oriented cycles of length 3. Then
$Q$ has a fully compatible positive semi-definite quasi-Cartan companion with a disjoint support companion basis.

\end{lemma}

\begin{proof}
We will use the same construction as in the proof of Theorem~\ref{thm fin mut type}, see Fig.~\ref{triang}.
Since we only have copies of block II in the decomposition, the arcs labelled by anything different from $e_i+e_j$ for some $i,j$ will only occur when two copies of block II are attached along two vertices to create a vanishing arrow. However, the corresponding vertex of the quiver is not a part of any triangle in the quiver (it is a vertex in an oriented quadrilateral).
By assumption, all  triangles in $Q$ are oriented, and as we see now, every arrow in a triangle is labelled by ``+''. Therefore the quasi-Cartan companion is fully compatible.

Furthermore, observe that any open vertex of the decomposition of $Q$ corresponds to an arc of a triangle $T_k$ such that two other arcs of $T_k$ lie at the boundary. Therefore, one arc of $T_k$ only corresponds to a vertex of $Q$, and thus the basis vector $e_k$ belongs to the support of the only vector corresponding to this open vertex, so we get  a disjoint support companion basis. 

%

\end{proof}

\begin{lemma}
\label{step}
Let $Q$ be a quiver decomposed into several copies of block II. Suppose that $Q$ is not the quiver  shown in Fig.~\ref{killhope}.  Then
$Q$ has a fully compatible positive semi-definite quasi-Cartan companion with a disjoint support companion basis.

\end{lemma}

\begin{proof}
To show the statement we will proceed by induction on the number of non-oriented triangles in $Q$. 




Lemma~\ref{no_nonor} constitutes the base of the induction (no non-oriented triangles in $Q$). Suppose that the statement is known for every quiver with less than $n$ non-oriented triangles, and consider a quiver $Q$  with $k$ non-oriented triangles.

Let $p,q,r$ be vertices of some non-oriented triangle in $Q$. Clearly, each of the edges $pq,qr,rp$ belongs to its own block of type II  (two of these blocks may have a second vertex in common), so the configuration of blocks forming the triangle $pqr$ looks like one of the three configurations shown in  Fig.~\ref{3blocks}, up to a symmetry obtained by changing the direction of all arrows (the orientations of edges of the non-oriented triangle determine all other arrows in the three blocks, furthermore two of the three open vertices may be attached to each other, this gives 4 possibilities, two of which coincide up to reversing all arrows). Denote this configuration by $C$. Notice that as shown in  Fig.~\ref{3blocks}, the configuration itself has a fully compatible positive semi-definite quasi-Cartan companion with disjoint support companion basis.

\begin{figure}[!h]
\begin{center}
\epsfig{file=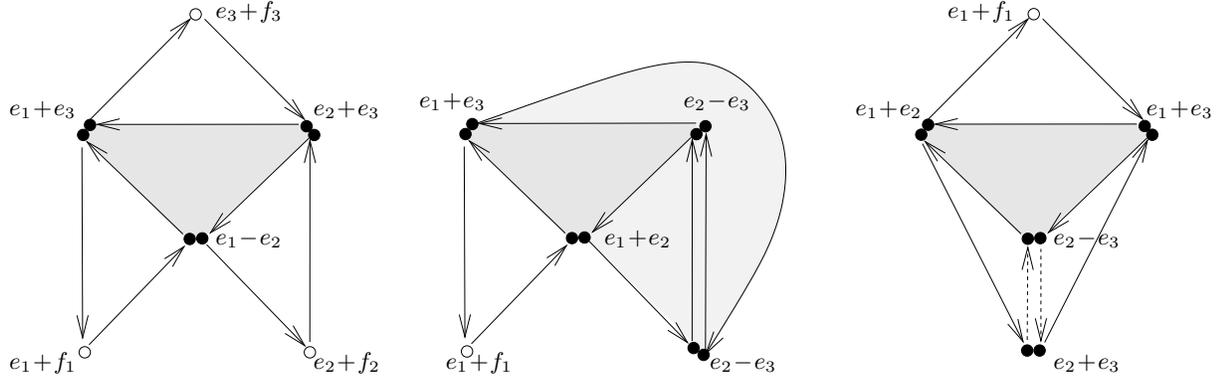,width=0.9\linewidth}
\put(-435,93){\scriptsize $e_1\!+\!e_3$}
\put(-320,93){\scriptsize $e_2\!+\!e_3$}
\put(-357,44){\scriptsize $e_1\!-\!e_2$}
\put(-357,130){\scriptsize $e_3\!+\!f_3$}
\put(-435,-3){\scriptsize $e_1\!+\!f_1$}
\put(-320,-3){\scriptsize $e_2\!+\!f_2$}
\put(-280,96){\scriptsize $e_1\!+\!e_3$}
\put(-180,96){\scriptsize $e_2\!-\!e_3$}
\put(-210,44){\scriptsize $e_1\!+\!e_2$}
\put(-270,-3){\scriptsize $e_1\!+\!f_1$}
\put(-170,-3){\scriptsize $e_2\!-\!e_3$}
\put(-80,130){\scriptsize $e_1\!+\!f_1$}
\put(-40,-3){\scriptsize $e_2\!+\!e_3$}
\put(-115,93){\scriptsize $e_1\!+\!e_2$}
\put(-5,93){\scriptsize $e_1\!+\!e_3$}
\put(-40,44){\scriptsize $e_2\!-\!e_3$}

\caption{Configurations containing a non-oriented triangle. }
\label{3blocks}
\end{center}
\end{figure}

An additional triangle attached to that configuration may be attached along one, two or three vertices.  
We will consider these three cases.

\medskip
\noindent
{\bf Case 1:}
First, suppose that every  triangle attached to $C$ is attached by at most one vertex. Then every edge of $Q$ either belongs to $C$ or two $Q\setminus C$ (here we understand $Q\setminus C$ as a subquiver spanned by all vertices contained in at least one block not lying in $C$). Notice that $Q\setminus C$ is a quiver from triangulation containing open vertices (so it is different from the quiver shown in  Fig.~\ref{killhope}).  Also  $Q\setminus C$ contains smaller number of non-oriented triangles than $Q$ (as it does not contain $pqr$). Hence, by the inductive assumption,  $Q\setminus C$  has a fully compatible positive semi-definite quasi-Cartan companion with a disjoint support companion basis.
Furthermore, since every triangle attached to $C$ is attached only along one vertex,  any two vertices in the intersection of $C$ and $Q\setminus C$ are not adjacent in $Q\setminus C$. By the inductive assumption, 
$C$ has a  fully compatible positive semi-definite quasi-Cartan companion with a disjoint support companion basis. By Lemma~\ref{disjoint}, this implies that the quiver $Q$ itself has  a  fully compatible positive semi-definite quasi-Cartan companion with  a disjoint support companion basis.

\medskip
\noindent
{\bf Case 2:}
Next, suppose that there is a triangle $T$ attached to $C$ along 2 vertices. Then the quiver $T \cup C$ will be one of the quivers shown in Fig.~\ref{2vert} (again, we identify configurations obtained by reversing all arrows).  Notice that each of these quivers has  a  fully compatible positive semi-definite quasi-Cartan companion with  a disjoint support companion basis (see Fig.~\ref{2vert}).
Furthermore, if no triangle is attached to two vertices of $T\cup C$ simultaneously, then the reasoning of Case 1 shows the statement for $Q$. 
If some triangle $T'$ is attached to the two open vertices of $T\cup C$, then the quiver  $T'\cup T\cup C$ also has
  fully compatible positive semi-definite quasi-Cartan companion with  a disjoint support companion basis (to see this one can identify $f:=f_1=f_2$ and assign the vector $f+f_0$ to the vertex of $T'$  which does not lie in $T\cup C$, see Fig.~\ref{2vert}). Then the same reasoning as in Case 1 shows the statement for $Q$.

\begin{figure}[!h]
\begin{center}

\epsfig{file=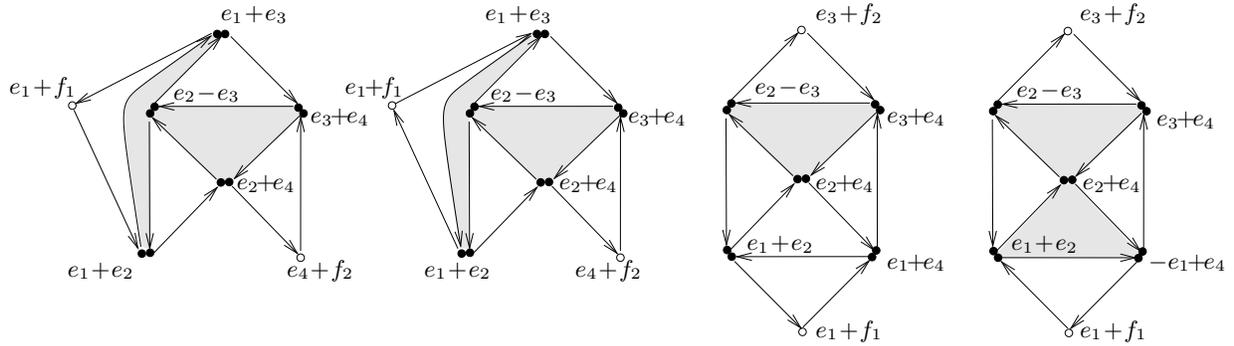,width=0.9\linewidth}
\put(-352,120){\scriptsize $e_1\!+\!e_3$}
\put(-318,81){\scriptsize $e_3\!\!+\!\!e_4$}
\put(-327,23){\scriptsize $e_4\!+\!f_2$}
\put(-432,92){\scriptsize $e_1\!+\!f_1$}
\put(-370,90){\scriptsize $e_2\!-\!e_3$}
\put(-410,23){\scriptsize $e_1\!+\!e_2$}
\put(-346,56){\scriptsize $e_2\!\!+\!\!e_4$}
\put(-305,92){\scriptsize $e_1\!\!+\!\!f_1$}
\put(-252,120){\scriptsize $e_1\!+\!e_3$}
\put(-275,23){\scriptsize $e_1\!+\!e_2$}
\put(-218,23){\scriptsize $e_4\!+\!f_2$}
\put(-224,56){\scriptsize $e_2\!\!+\!\!e_4$}
\put(-198,81){\scriptsize $e_3\!\!+\!\!e_4$}
\put(-250,90){\scriptsize $e_2\!-\!e_3$}
\put(-127,0){\scriptsize $e_1\!+\!f_1$}
\put(-127,120){\scriptsize $e_3\!+\!f_2$}
\put(-153,33){\scriptsize $e_1\!+\!e_2$}
\put(-100,26){\scriptsize $e_1\!\!+\!\!e_4$}
\put(-127,56){\scriptsize $e_2\!\!+\!\!e_4$}
\put(-100,81){\scriptsize $e_3\!\!+\!\!e_4$}
\put(-150,92){\scriptsize $e_2\!-\!e_3$}
\put(-27,0){\scriptsize $e_1\!+\!f_1$}
\put(-27,120){\scriptsize $e_3\!+\!f_2$}
\put(-53,33){\scriptsize $e_1\!+\!e_2$}
\put(-1,26){\scriptsize $-e_1\!\!+\!\!e_4$}
\put(-26,56){\scriptsize $e_2\!\!+\!\!e_4$}
\put(2,80){\scriptsize $e_3\!\!+\!\!e_4$}
\put(-51,91){\scriptsize $e_2\!-\!e_3$}

\caption{Triangle $T$ attached to $C$ by two vertices. }
\label{2vert}
\end{center}
\end{figure}

\medskip
\noindent
{\bf Case 3:} Finally, suppose that there is a triangle $T$ attached to $C$ along 3 vertices.
Then the quiver $T \cup C$ will be one of the quivers shown in Fig.~\ref{3vert}. The first of these quivers coincides with the exceptional quiver shown  in Fig.~\ref{killhope} and has no fully compatible quasi-Cartan companion. The second one has a fully compatible quasi-Cartan companion (see Fig~\ref{3vert} for the corresponding vectors). It has no open vertices, so it cannot be a part of any larger quiver from triangulation.

\begin{figure}[!h]
\begin{center}
\epsfig{file=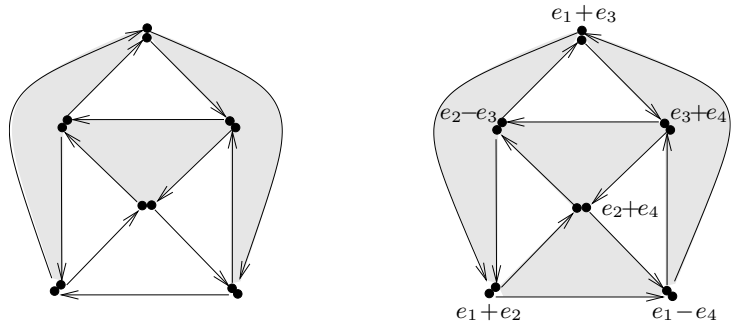,width=0.6\linewidth}
\put(-68,107){\scriptsize $e_1\!+\!e_3$}
\put(-104,-6){\scriptsize $e_1\!+\!e_2$}
\put(-30,-6){\scriptsize $e_1\!-\!e_4$}
\put(-49,33){\scriptsize $e_2\!\!+\!\!e_4$}
\put(-110,69){\scriptsize $e_2\!\!-\!\!e_3$}
\put(-23,69){\scriptsize $e_3\!\!+\!\!e_4$}
\caption{Triangle attached to $C$ by three vertices. }
\label{3vert}
\end{center}
\end{figure}

\end{proof}



\begin{proof}[Proof of Theorem~\ref{thm fully}:]

To complete the proof of Theorem~\ref{thm fully}, it is sufficient to show that there is a required quasi-Cartan companion for block decompositions containing blocks other than block of type II.
We will first show that one can add blocks of type I and IV and then that one can add blocks of types  IIIa, IIIb or V.

For the blocks of types I and IV, we can substitute such a block in the block decomposition by a block of type II (of course, this slightly changes the surface), we will obtain some different quiver $Q'$. After this substitution we will never obtain the quiver shown in Fig.~\ref{killhope} (as this quiver does not contain open vertices while the process of substitution does introduce such vertices). So, $Q'$ has a required quasi-Cartan companion with a disjoint support realisation. To get a quasi-Carton companion for $Q$, we just remove extra vertex for the case of the block of type I, and add an additional vector as in Fig.~\ref{bl_I_IV} for the case of the block of type IV.

Finally, to treat also the blocks of  types  IIIa, IIIb and V, let $Q_1$ be the union of such blocks.
Let $Q_2$ be the union of all other blocks in $Q$.
As each block of  type  IIIa, IIIb or V has a unique open vertex, the subquiver $Q_1\cap Q_2$ has no arrows.
Furthermore, as $Q_2$ is a quiver without blocks of type  IIIa, IIIb and V, it has a fully compatible positive
semi-definite quasi-Cartan companion with a disjoint support companion basis. In view of Lemma~\ref{bl_}, the quiver $Q_1$ also has a fully compatible positive semi-definite quasi-Cartan companion with a disjoint support companion basis.
So, by Lemma~\ref{disjoint} the quiver $Q$ itself  has a fully compatible positive semi-definite quasi-Cartan companion with a disjoint support companion basis, as required.

\end{proof}

\begin{figure}[!h]
\begin{center}
\epsfig{file=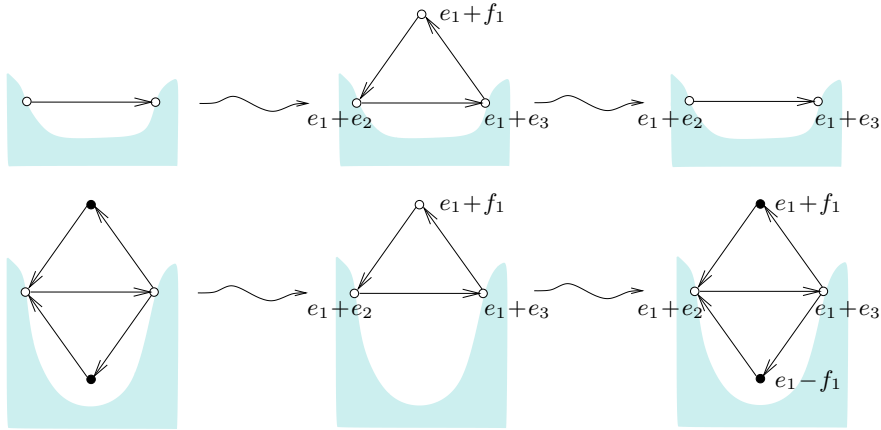,width=0.7\linewidth}
\put(-205,115){\scriptsize $e_1\!+\!e_2$}
\put(-138,115){\scriptsize $e_1\!+\!e_3$}
\put(-155,155){\scriptsize $e_1\!+\!f_1$}
\put(-80,115){\scriptsize $e_1\!+\!e_2$}
\put(-13,115){\scriptsize $e_1\!+\!e_3$}
\put(-205,43){\scriptsize $e_1\!+\!e_2$}
\put(-138,43){\scriptsize $e_1\!+\!e_3$}
\put(-155,83){\scriptsize $e_1\!+\!f_1$}
\put(-80,43){\scriptsize $e_1\!+\!e_2$}
\put(-13,43){\scriptsize $e_1\!+\!e_3$}
\put(-28,83){\scriptsize $e_1\!+\!f_1$}
\put(-28,16){\scriptsize $e_1\!-\!f_1$}

\caption{Constructing companions for quivers with blocks of type I and IV.} 
\label{bl_I_IV}
\end{center}
\end{figure}

\begin{remark}
A mutation of a fully compatible companion is not necessarily a fully compatible quasi-Cartan companion,
see Fig.~\ref{ex_puncture}.

\end{remark}

\begin{figure}[!h]
\begin{center}
\epsfig{file=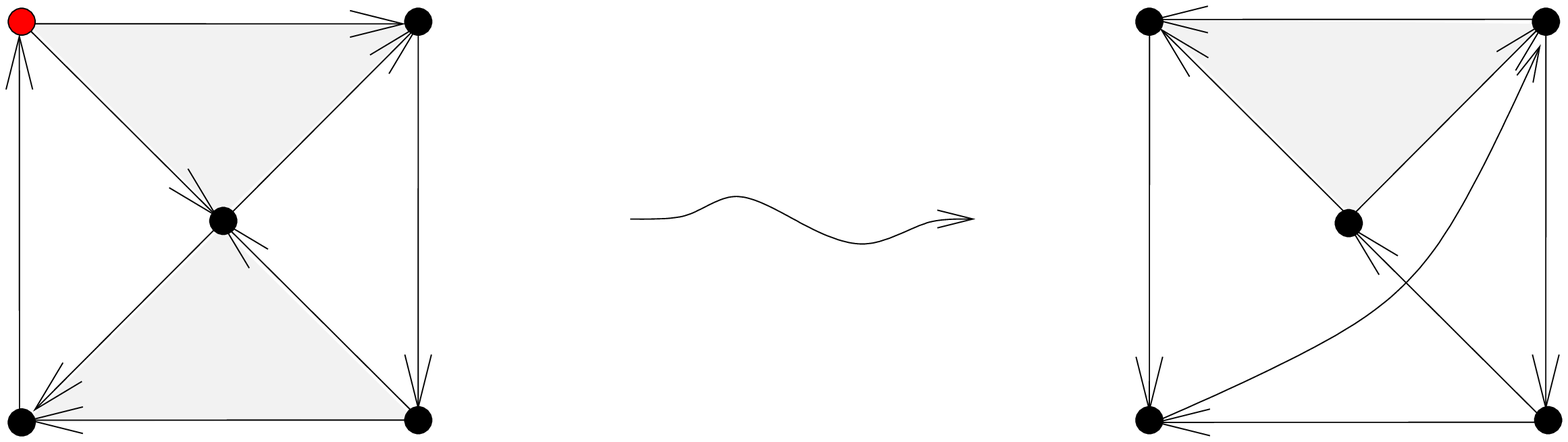,width=0.6\linewidth}
\put(-150,50){$\mu_1$}
\put(-270,78){$1$}
\put(-310,70){\small $e_3-e_1$}
\put(-310,0){\small $e_4-e_1$}
\put(-195,70){\small $e_2+e_3$}
\put(-195,0){\small $e_2+e_4$}
\put(-269,35){\small $e_1\!+\!e_2$}
\put(-110,70){\small $e_1-e_3$}
\put(-110,0){\small $e_4-e_3$}
\put(5,70){\small $e_2+e_3$}
\put(5,0){\small $e_2+e_4$}
\put(-68,30){\small $e_1\!+\!e_2$}
\caption{Mutation of a fully compatible quasi-Cartan companion is not always fully compatible: after the mutation there is an oriented triangle labeled with vectors $e_2+e_3$, $e_2+e_4$ and $e_4-e_3$.}
\label{ex_puncture}
\end{center}
\end{figure}

\section{Group from an (unpunctured) surface quiver}
\label{group}

The construction described in this section was initiated in~\cite{BM} for the case of quivers of finite type  and then extended in~\cite{FeTu}
for affine type quivers, quivers from unpunctured surfaces and exceptional mutation-finite quivers.

\subsection{Construction of the group}
\label{group-constr}

Here we present the construction for the case of unpunctured surface. 

Given a quiver $Q$ from an unpunctured surface, we construct a group $G=G(Q)$ as follows:
\begin{itemize}
\item the generators $s_1,\dots,s_n$ of $G$ correspond to the vertices of $Q$;
\item there are five types of relations:
\begin{itemize}
\item[(R1)] $s_i^2=e$ for all $i=1,\dots,n$;
\item[(R2)] $(s_is_j)^{m_{ij}}= e$  for all $i,j$ not joined by double arrow, where
$$
m_{i,j}=\begin{cases}
2,& \text{if $i,j,$ are not joined}; \\
3,& \text{if $i,j,$ are  joined by a single arrow};
\end{cases} 
$$  
\item[(R3)] (cycle relations) \\ $(s_1\ s_2s_3s_2)^2=e$ for every subquiver of $Q$ shown in Fig.~\ref{rel}(a) and \\
 $(s_1\ s_2s_3s_2)^3=e$ for every subquiver of $Q$ shown in Fig.~\ref{rel}(b)  respectively;
\item[(R4)]  ($\widetilde A_3$-relation) \\ $(s_1\ s_2s_3s_4s_3s_2)^2=e$ for every subquiver of $Q$ shown in Fig.~\ref{rel}(c);
\item[(R5)]  (handle relations) \\
 $(s_1\ s_2s_3s_4s_3s_2)^3=e$ for every subquiver of $Q$ shown in Fig.~\ref{rel}(d); \\
$(s_1\ s_2s_3s_4s_5s_4s_3s_2)^2=e$ for every subquiver of $Q$ shown in Fig.~\ref{rel}(e).

\end{itemize}

\end{itemize}

\begin{figure}[!h]
\begin{center}
\epsfig{file=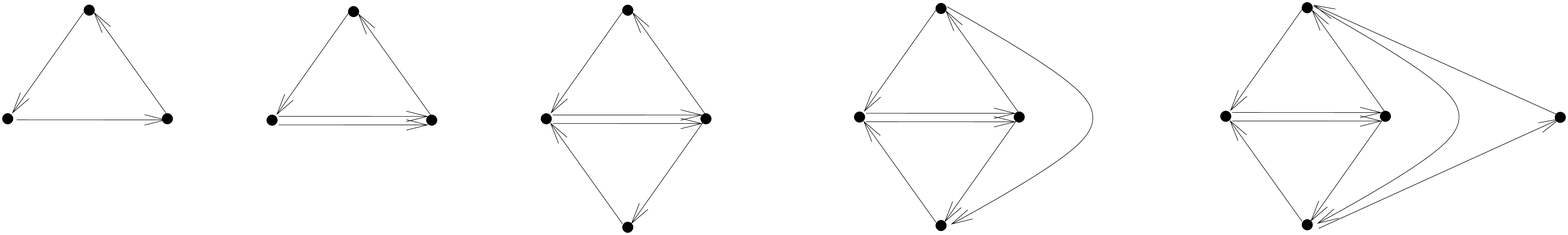,width=0.99\linewidth}
  \put(-458,25){\scriptsize $1$}
\put(-433,65){\scriptsize $2$}
\put(-400,25){\scriptsize $3$}
\put(-380,25){\scriptsize $1$}
\put(-358,65){\scriptsize $2$}
\put(-323,25){\scriptsize $3$}
\put(-300,25){\scriptsize $1$}
\put(-280,65){\scriptsize $2$}
\put(-245,25){\scriptsize $3$}
\put(-280,-3){\scriptsize $4$}
\put(-210,25){\scriptsize $1$}
\put(-190,65){\scriptsize $2$}
\put(-155,25){\scriptsize $3$}
\put(-190,-3){\scriptsize $4$}
\put(-105,25){\scriptsize $1$}
\put(-85,65){\scriptsize $2$}
\put(-50,25){\scriptsize $3$}
\put(-85,-3){\scriptsize $4$}
\put(1,25){\scriptsize $5$}
\put(-460,-43){\small (a)}
\put(-380,-43){\small (b)}
\put(-300,-43){\small (c)}
\put(-210,-43){\small (d)}
\put(-105,-43){\small (e)}
\put(-460,-23){\scriptsize  $(s_1\ s_2s_3s_2)^2=e$ }
\put(-380,-23){\scriptsize  $(s_1\ s_2s_3s_2)^3=e$}
\put(-300,-23){\scriptsize  $(s_1\ s_2s_3s_4s_3s_2)^2=e$ }
\put(-210,-23){\scriptsize  $(s_1\ s_2s_3s_4s_3s_2)^3=e$ }
\put(-105,-23){\scriptsize $(s_1\ s_2s_3s_4s_5s_4s_3s_2)^2=e$}
\put(-430,93){  Cycle relations: }
\put(-300,93){\normalsize  $\widetilde A_3$-relation: }
\put(-170,93){ Handle relations:}
\caption{Relations R3,R4,R5 for the group $G$.  }
\label{rel}
\end{center}
\end{figure}

\begin{theorem}[\cite{BM},\cite{FeTu}]
\label{G invar}
If $Q$ is a quiver arising from an unpunctured surface 
and $G=G(Q)$ 
is a group defined as above, then $G$
is invariant under the mutations of $Q$.

\end{theorem}

\begin{remark}
If $Q$ is a quiver and $\mu_k(Q)$ is a mutation of $Q$ in the direction $k$, then the isomorphism of groups $G_1=G(Q)$ and $G_2=G(\mu_k(Q))$ 
can be described as follows. If $\{s_i\}$ and $\{t_i\}$ are the generators of $G(Q)$ and $G(\mu_k(Q))$  described above, then
$$
t_i=\begin{cases}
s_ks_is_k,& \text{if $Q$ contains an arrow from $i$ to $k$};\\
s_i,      & \text{ otherwise}.      
\end{cases}
$$

\end{remark}

\begin{remark}
Theorem~\ref{G invar} implies that the group $G$ does not depend on the choice of triangulation of the corresponding surface $S$, so one can say that $G=G(S)$ is the group assigned to the topological surface $S$. 

\end{remark}

\begin{remark}
Relations (R1),(R2) define a Coxeter group, the other relations turn $G$ into a quotient of a Coxeter group.
However, in the cases of finite and affine quivers, by choosing an acyclic representative $Q$ one can see that there are no non-Coxeter relations in $G$, so $G$ is a Coxeter group itself. Applying Theorem~\ref{G invar} we see that $G$ is a Coxeter group for any quiver of finite or affine type.  

\end{remark}

\subsection{Moving the marked points from one boundary component to another}

In Section~\ref{group-constr} we recalled the construction of a group for any quiver from unpunctured surface $S$.
It is natural to ask whether the group $G=G(S)$ uniquely defines the surface $S$.
The main result of this section indicates that it does not: one can move boundary marked points from one boundary component to another without changing the group.

\begin{theorem}
\label{boundaries}
Let $S_{g,b}$ be an unpunctured surface of genus $g$  with $b$ boundary components. Then the group $G(S_{g,b})$ does not depend on the distribution of the boundary marked points along the boundary components of the surface (depending only on $g,b$ an the total number of boundary marked points). 

\end{theorem}

Denote by $G(S_{g,b};k_1,\dots,k_b)$ the group constructed from  $S_{g,b}$  with $k_1,\dots,k_b$ marked points on the boundary components, $k_i\ge 1$.

\begin{lemma}
\label{l 13-22}
$G(S_{0,3};1,1,3)\cong G(S_{0,3};1,2,2)$.

\end{lemma}

\begin{proof}
Let $S_1$ and $S_2$ be the surfaces with $(1,1,3)$ and $(1,2,2)$ boundary marked points respectively. Consider the triangulations of $S_1$ and $S_2$ as in Fig.~\ref{13-22}.
Let $s_i$ be the generators for $G(S_1)$ corresponding to the triangulation, and let $t_i$ be the set of generators of $G(S_1)$ satisfying
$$
t_i=
\begin{cases}
s_4s_5s_4, & \text{if $i=5$,}\\
s_i, & \text{otherwise.}
\end{cases}
$$ 
We will show that the defining relations for  $G(S_1)$ in terms of $t_i$ will rewrite exactly as defining relations for $G(S_2)$ in the triangulation shown in  
Fig.~\ref{13-22} with $t_i$ corresponding to the $i$--th arc. Indeed, all the relations (except for ones corresponding to Coxeter relations and cycle relations for arcs in the dark quadrilaterals) coincide, and  all  other relations can be checked directly.
For example, the Coxeter relation $(s_5s_6)^2=e$ will turn into the cycle relation $( t_4t_5t_4\ t_6)^2=e$, while the cycle relation   $(s_3\ s_4s_5s_4)^2=e$ turns into a Coxeter relation $ (t_3t_5)^2=e$.
Moreover, every defining relation for  $G(S_2)$ is  obtained from the defining relations for $S_1$ in this way, which implies that the groups coincide.

\end{proof}

\begin{figure}[!h]
\begin{center}
\epsfig{file=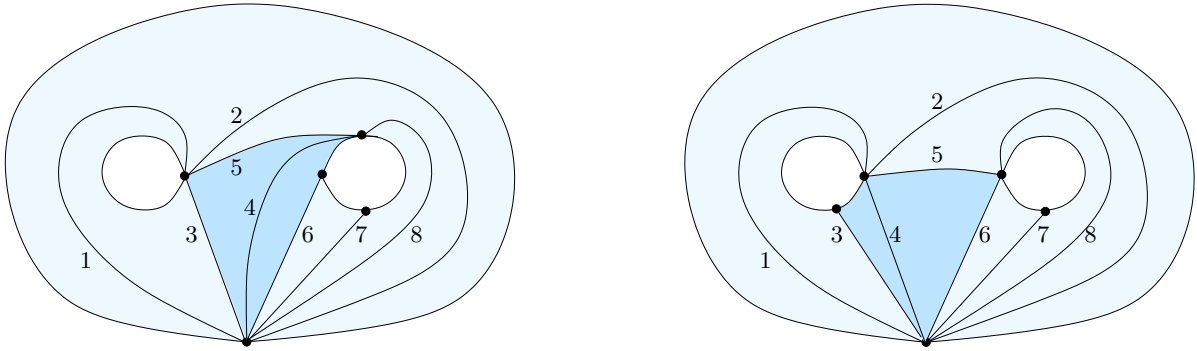,width=0.99\linewidth}
\put(-422,30){\scriptsize $1$}
\put(-382,40){\scriptsize $3$}
\put(-360,50){\scriptsize $4$}
\put(-365,65){\scriptsize $5$}
\put(-338,40){\scriptsize $6$}
\put(-318,40){\scriptsize $7$}
\put(-297,40){\scriptsize $8$}
\put(-365,85){\scriptsize $2$}
\put(-165,30){\scriptsize $1$}
\put(-138,40){\scriptsize $3$}
\put(-116,40){\scriptsize $4$}
\put(-100,70){\scriptsize $5$}
\put(-82,40){\scriptsize $6$}
\put(-60,40){\scriptsize $7$}
\put(-42,40){\scriptsize $8$}
\put(-100,90){\scriptsize $2$}

\caption{Moving a boundary marked point form one boundary component to another.}
\label{13-22}
\end{center}
\end{figure}

\begin{lemma}
\label{insert}
$G(S_{0,3};1,1,m)=G(S_{0,3};1,2,m-1)$ for any $m\ge 2$.

\end{lemma}
\begin{proof}
To show the statement we will adjust slightly the proof of  Lemma~\ref{l 13-22}. Namely, we will insert 
one more arc between the arcs labeled $7$ and $8$ in both parts of Fig.~\ref{13-22} without changing anything else in the proof (see Fig.~\ref{13-22ins}, left). Repeating this several times we see that  
$G(S_{0,3},1,1,m)=G(S_{0,3},1,2,m-1)$ for all  $m\ge 3$. For $m=2$ the statement holds  trivially.

\end{proof}

\begin{figure}[!h]
\begin{center}
\epsfig{file=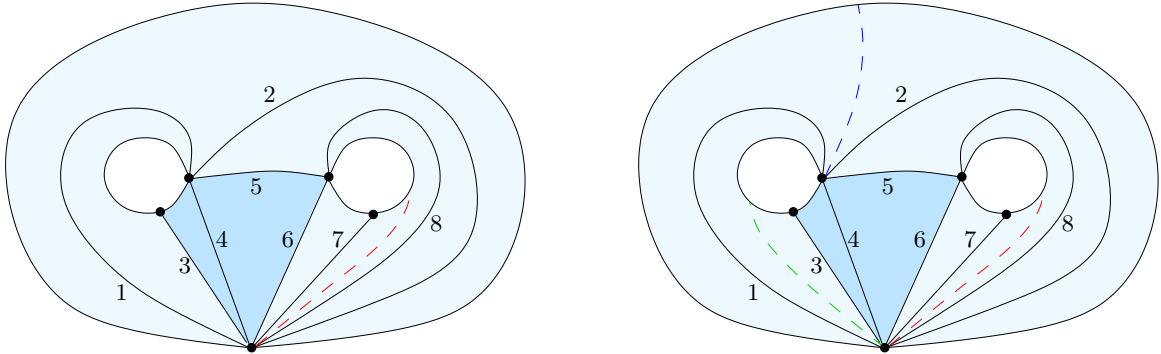,width=0.96\linewidth}
\put(-156,20){\scriptsize $1$}
\put(-132,30){\scriptsize $3$}
\put(-118,40){\scriptsize $4$}
\put(-105,60){\scriptsize $5$}
\put(-93,40){\scriptsize $6$}
\put(-74,40){\scriptsize $7$}
\put(-37,46){\scriptsize $8$}
\put(-100,95){\scriptsize $2$}
\put(-395,20){\scriptsize $1$}
\put(-371,30){\scriptsize $3$}
\put(-357,40){\scriptsize $4$}
\put(-344,60){\scriptsize $5$}
\put(-332,40){\scriptsize $6$}
\put(-313,40){\scriptsize $7$}
\put(-276,46){\scriptsize $8$}
\put(-339,95){\scriptsize $2$}

\caption{Inserting more boundary marked points to each of the three boundary components.}
\label{13-22ins}
\end{center}
\end{figure}

\begin{lemma}
If $k+l+m=k'+l'+m'$ then $G(S_{0,3};k,l,m)=G(S_{0,3};k',l',m')$.

\end{lemma}

\begin{proof}
We will prove that $G(S_{0,3};k,l,m)=G(S_{0,3};k,l+1,m-1)$ for any $k,l\ge 1$, $m\ge 2$. For this we modify the reasoning in
the proof of  Lemma~\ref{l 13-22} once again: in Lemma~\ref{insert} we increased $m$ by inserting a new arc between arcs $7$ and $8$; similarly, 
we can increase number $k$ by inserting an arc between the arcs labelled $1$ and $2$, and  we can increase the number $l$ by inserting an arc between the arcs labelled $1$ and $3$ (see Fig.~\ref{13-22ins}, right). 

So, we can move a boundary marked point from any boundary component to any other boundary component.
\end{proof}

\begin{lemma}
\label{3}
If $k+l+m=k'+l'+m'$ and $b\ge 3$ then $G(S_{g,b};k,l,m)=G(S_{g,b};k',l',m')$.

\end{lemma}

\begin{proof}
Consider once more Fig.~\ref{13-22} and the proof of  Lemma~\ref{l 13-22}.
We can increase the genus or the number of boundary components of the surface $S$ by attaching a
triangulated surface $S_{1,1}$ (a torus with 2 boundary marked points on a unique boundary component) or $S_{0,2}$ (an annulus with 1 marked point on one boundary component and 2 marked points on the other).
We will attach this small surface  along one of its boundary segments to the boundary segment of $S$ lying in  Fig.~\ref{13-22} between arcs $7$ and $8$. As this does not affect the shaded regions, this will not affect the proof.

\end{proof}

It is left to consider the case when we cannot choose three boundary components like in  Lemma~\ref{l 13-22},
i.e. the case when we only have 2 boundary components (in case of $b=1$ there is nothing to prove).

\begin{lemma}
\label{2}
If $k+l=k'+l'$  then $G(S_{g,2};k,l)=G(S_{g,2};k',l')$.

\end{lemma}

\begin{proof}
The proof is by induction on $g$. The base, $g=0$, is known: in this case we deal with an annulus, and the group $G$ is the affine Weyl group $\widetilde A_{k+l}$. 
Assume that the lemma is known for all surfaces of genus $g$.

To increase the genus of $S=S_{g,2}$ , we cut $S$ along any arc $\alpha$ connecting two boundary components and insert a handle $S_{1,1}$ 
(with two marked points on the boundary) between the sides $\alpha_1$ and $\alpha_2$ of the cut as in Fig.~\ref{torus}.
As a result, we obtain a surface $S'=S_{g+1,2}$. 
We then choose a triangulation of the handle so that the arcs $\alpha_1$ and $\alpha_2$ are not adjacent
(for this it is sufficient to include the arc $\beta$ separating $\alpha_1$ from $\alpha_2$ at both ends and going through the handle in between).

Cutting $S'$ along $\alpha_1$ and $\alpha_2$, we obtain two connected components: a torus $S_{1,1}$ and some other surface $P$ of genus $g$.
Let $G(P)$ be the corresponding group for this surface.
Then $G(S')$ is an amalgamated product of $G(P)$ and $G(S_{1,1})$ along the common subgroup 
$\langle s_{\alpha_1},s_{\alpha_2}\ | \ (s_{\alpha_1}s_{\alpha_2})^2=e\rangle$.
As neither  $G(P)$ nor $G(S_{1,1})$ depends on the distribution of the boundary marked points, we get the lemma.

\end{proof}

\begin{remark}
While considering block decompositions of $P$ and $S_{1,1}$, we glue additional two triangles along $\alpha_1$ and $\alpha_2$ to each of them for the arcs $\alpha_1$ and $\alpha_2$ to be interior ones (and thus to correspond to generators of the groups). When gluing surfaces together we remove these four triangles.
  \end{remark}

\begin{figure}[!h]
\begin{center}
\epsfig{file=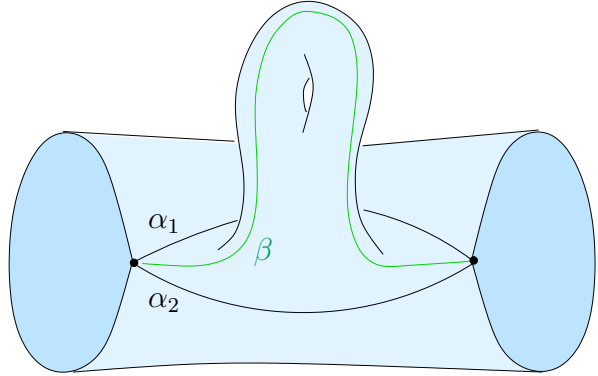,width=0.49\linewidth}
\put(-170,55){$\alpha_1$}
\put(-170,25){$\alpha_2$}
\put(-130,43){\color{ForestGreen} $\beta$}
\caption{Inserting an extra handle. }
\label{torus}
\end{center}
\end{figure}

Lemma~\ref{2} together with  Lemma~\ref{3} prove  Theorem~\ref{boundaries}.

\section{Extended affine Weyl group for a surface}
\label{eawg}
In this section we provide a construction of another group $W(S)$ from a bordered marked unpunctured surface $S$, and then explore its relations with $G(S)$ defined above.

\subsection{Extended affine Weyl group}
We remind the definition of extended affine Weyl groups following~\cite{MS,AS} (these groups are called ``toroidal Weyl groups'' in~\cite{MS}).

Let $V$ be a quadratic space with a quadratic form $\langle\cdot,\cdot\rangle$ of signature $(n_+,n_0)$, where $n_0$ is a dimension of the radical $V_0$. Choose a maximal positive-definite subspace $V_+$, i.e. $dim V_+=n_+$ and $V= V_+\oplus V_0$.

An {\it extended affine root system} $R$ is a set of roots (vectors) in $V$, such that $R$ is discrete, indecomposable, reduced and closed under reflections with respect to the hyperplanes orthogonal to the real roots in $R$ (for detailed definitions and properties we refer to~\cite{Sa1,AABGP}). 

Choose bases $\{v_1,\dots,v_{n_+}\}$ and $\{\delta_1,\dots,\delta_{n_0}\}$ in $V_+$ and $V_0$ respectively, and consider the space  $V\oplus V_0^*= V_+\oplus V_0 \oplus V_0^*$ with basis $\{v_1,\dots,v_{n_+},\delta_1,\dots,\delta_{n_0}, \delta_1^*, \dots, \delta_{n_0}^* \}$. We extend the quadratic form  $\langle\cdot,\cdot\rangle$ to $V\oplus V_0^*$ by 
$$
\langle v_i, \delta_j^*\rangle=0\,\forall i=1,\dots,n_+,\qquad\qquad
\langle \delta_i, \delta_j^*\rangle=
\begin{cases}
1,& \text{if $i=j$ };\\
0,& \text{otherwise}
\end{cases}  
$$
Then one can consider the action of the reflections in the real roots of $R$ in $V\oplus V_0^*$.
The {\em extended affine Weyl group} $W=W(R)$ is the group acting in $V\oplus V_0^*$ generated by reflections in real roots of $R$.

\subsection{Construction of a special triangulation.}
\label{triang-sp}
We now take an unpunctured surface $S$ with a particular choice of triangulation as described below. 
For this triangulation we consider the corresponding quiver $Q$ and construct a positive semi-definite fully compatible quasi-Cartan companion. The  reflections in the companion basis will generate an extended affine Weyl group of type $A_{n_+}^{[n_0]}$ for certain $n_+,n_0$.

An unpunctured surface $S$ contains the following features: boundary components (each with a number of boundary marked points) and handles. To construct the triangulation we do the following:

\begin{itemize}
\item[--] Choose any boundary component $b_0$ and a marked point $p$ on it.
\item[--] Consider three arcs (loops) $x,y,z$ as in Fig.~\ref{tr}, top:
\begin{itemize}
\item[-] all three of them have both ends at $p$, 
\item[-] $x$ separates the boundary component $b_0$; 
\item[-] $y$ separates all other boundary components; 
\item[-] $z$ separates all handles. 
\end{itemize}
\item[--] triangulate the region separated by $x$ as  in Fig.~\ref{tr}, bottom left. We can schematically show the quiver   corresponding quiver as in  Fig.~\ref{tr}, bottom right.
\item[--] Triangulate the regions separated by arcs $y$ and $z$ as in Fig.~\ref{tr3}. Use the triangulations of handles and of annuli  shown in the middle of  Fig.~\ref{tr3} (notice, that the triangulation of the annulus includes a  region triangulated as the domain separated by $x$, cf. Fig.~\ref{tr}, bottom left.)

\end{itemize}

\begin{figure}[!h]
\begin{center}
\epsfig{file=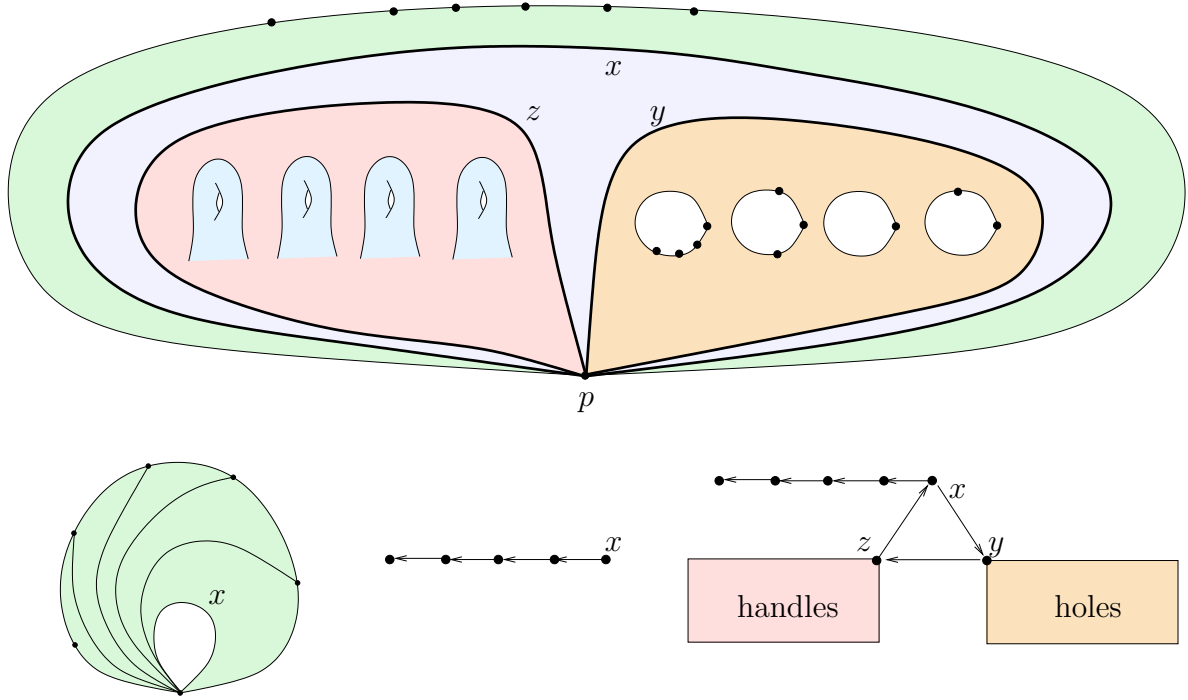,width=0.98\linewidth}
\put(-220,235){$x$}
\put(-250,218){$z$}
\put(-203,218){$y$}
\put(-90,75){$x$}
\put(-75,55){$y$}
\put(-230,110){$p$}
\put(-125,55){$z$}
\put(-370,35){$x$}
\put(-220,55){$x$}
\put(-170,30){handles}
\put(-50,30){holes}
\caption{Constructing the triangulation }
\label{tr}
\end{center}
\end{figure}

\begin{figure}[!h]
\begin{center}
\epsfig{file=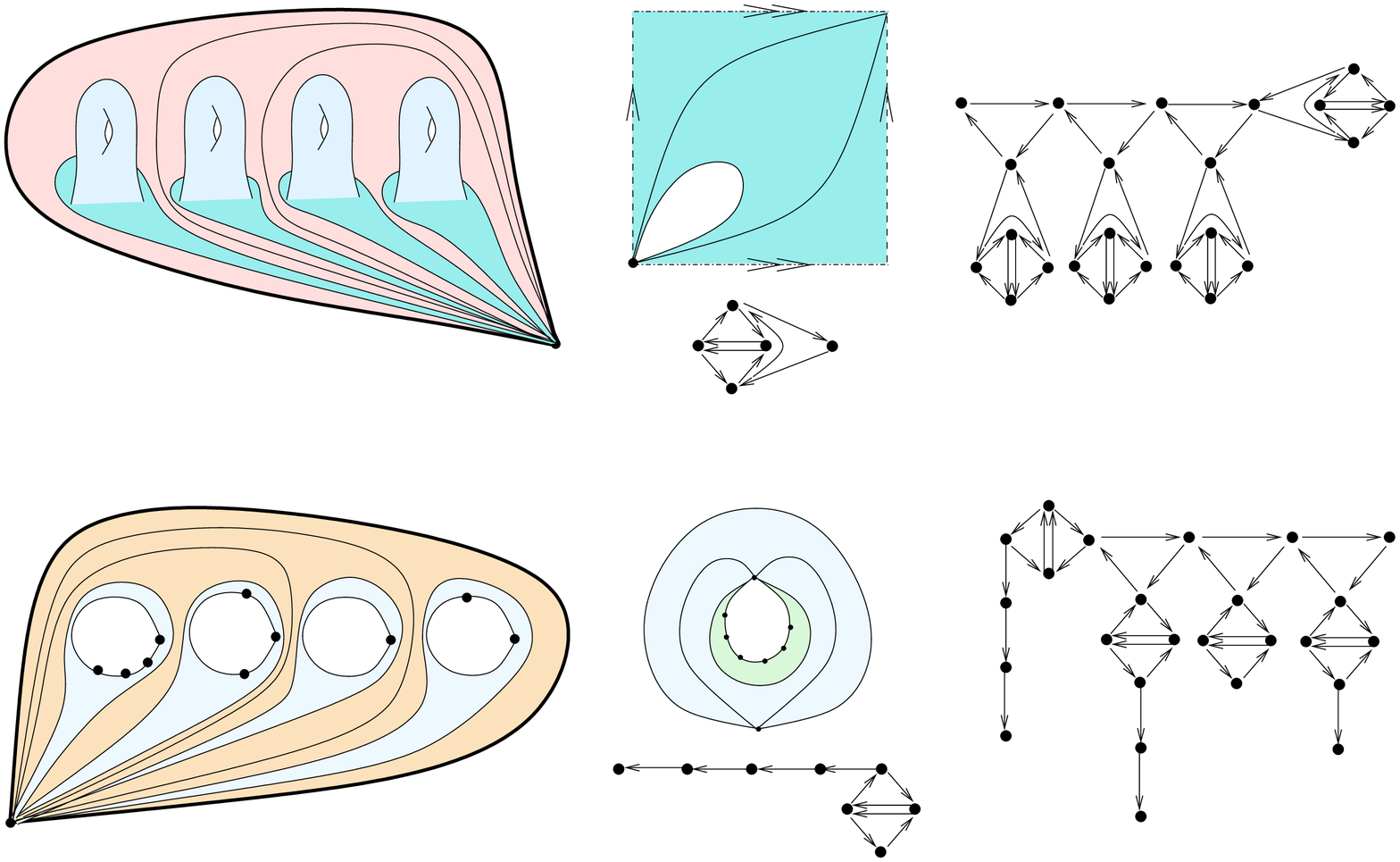,width=0.98\linewidth}
\put(-415,265){$z$} 
\put(-438,219){\scriptsize $z_1$} 
\put(-397,219){\scriptsize $z_2$} 
\put(-363,219){\scriptsize $z_3$} 
\put(-330,219){\scriptsize $z_4$} 
\put(-177,165){$z_i$} 
\put(-210,220){$z_i$}
\put(-145,245){$z$}
\put(-135,219){\scriptsize $z_1$} 
\put(-105,219){\scriptsize $z_2$} 
\put(-70,219){\scriptsize $z_3$} 
\put(-50,245){\scriptsize $z_4$} 
\put(-315,110){$y$} 
\put(-405,90){\scriptsize $y_4$} 
\put(-385,90){\scriptsize $y_3$} 
\put(-350,90){\scriptsize $y_2$} 
\put(-313,90){\scriptsize $y_1$} 
\put(-315,110){$y$} 
\put(-1,110){$y$} 
\put(-235,110){$y_i$} 
\put(-160,-3){$y_i$} 
\put(-100,107){\scriptsize $y_4$} 
\put(-77,80){\scriptsize $y_3$} 
\put(-47,80){\scriptsize $y_2$} 
\put(-13,80){\scriptsize $y_1$}  
\caption{Constructing the triangulation, cont.: regions with handles (top) and holes (bottom). }
\label{tr3}
\end{center}
\end{figure}

The quiver corresponding to the constructed triangulation  is shown in Fig.~\ref{tr-q}.

\begin{figure}[!h]
\begin{center}
\epsfig{file=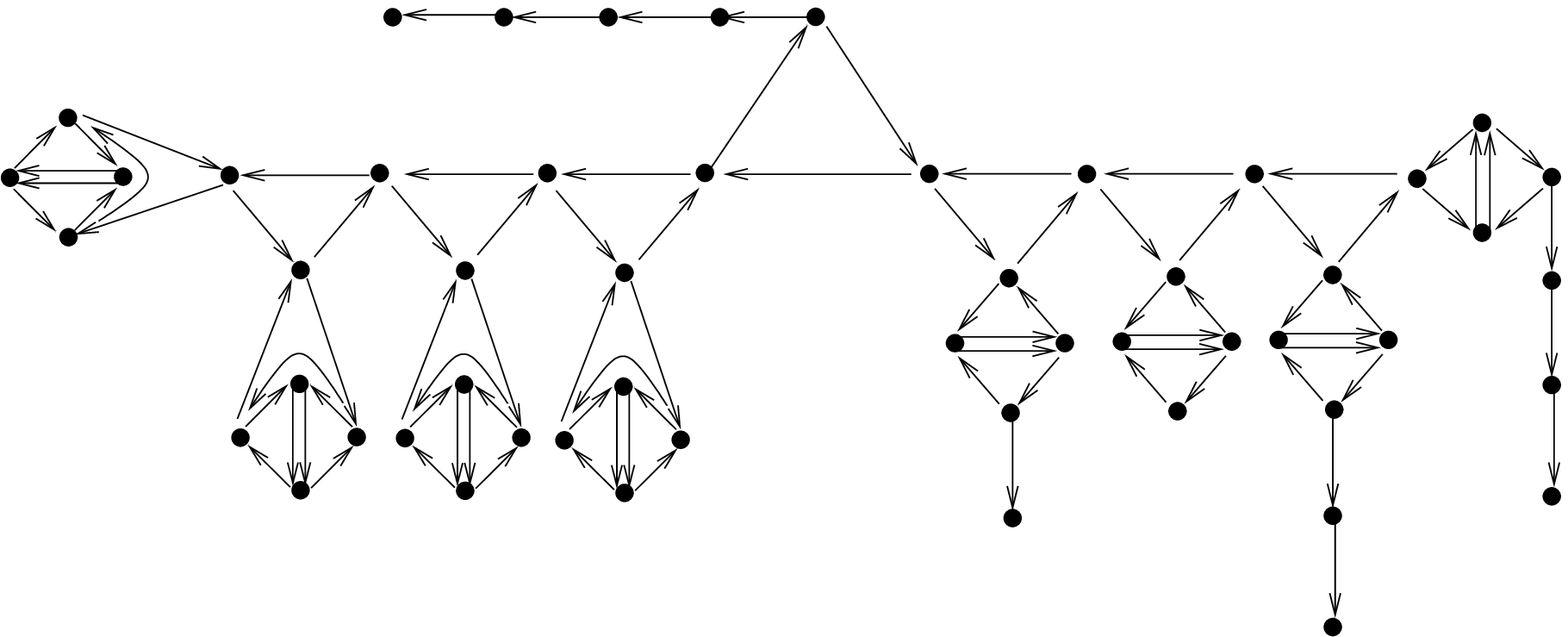,width=0.95\linewidth}
\put(-205,175){$x$}
\put(-175,135){$y$}
\put(-240,135){$z$}
\caption{The quiver }
\label{tr-q}
\end{center}
\end{figure}

\medskip
\noindent
{\bf Construction of the root system.}
Next, we construct a companion basis by assigning vectors to the nodes of $Q$. For this, we first consider a subquiver $Q'$ in $Q$ obtained by removing from $Q$ all nodes labelled by squares  in  Fig.~\ref{tr-q-vect}.
Notice that $Q'$ is a quiver of type $A_m$ for some $m$ (see e.g.~\cite{H}).
Label the nodes of $Q'$ by roots in the root system of type $A_m$ (more precisely, of type $A_{n-2g-b+1}$, where $n$ is the number of nodes in $Q$, $g$ is the genus of the surface and $b$ is the number of boundary components) providing a fully compatible quasi-Cartan companion for $Q'$. We denote the vector space spanned by the companion basis by $V_+$.

We now extend the quadratic form to $n$-dimensional vector space by adding $(2g+b-1)$-dimensional radical $V_0$. Having done that, we assign to the remaining nodes of $Q$ the vectors as in  Fig.~\ref{tr-q-vect}: these additional vectors are constructed from the vectors associated to the adjacent nodes and basis vectors of the radical. More precisely, for each boundary component we have one radical vector ($\delta_k$ in  Fig.~\ref{tr-q-vect}) and for each handle we have two radical vectors ($\delta_{ij}^{1,2}$ in  Fig.~\ref{tr-q-vect}). Let $u_1,\dots,u_n$ be the vectors obtained, denote  ${\bf u}=\{u_1,\dots,u_n \}$. These vectors define a positive semi-definite fully compatible quasi-Cartan companion of $Q$.

\medskip
Let $r_i=r_{u_i}$ be the reflections with respect to vectors $u_i$, denote by $W=W({\bf u},Q)$ the group generated by reflections $r_i=r_{u_i}$ acting on $V_+\oplus V_0$. We can also consider the action of $W$ on the $V_+\oplus V_0\oplus V_0^*$. By construction, $W$ is an extended affine Weyl group of type $A_{n_+}^{[n_0]}$, where the signature of the corresponding quadratic form is given by $(n_+,n_0)=(n-2g-b+1,2g+b-1)$.

\begin{figure}[!h]
\begin{center}
\epsfig{file=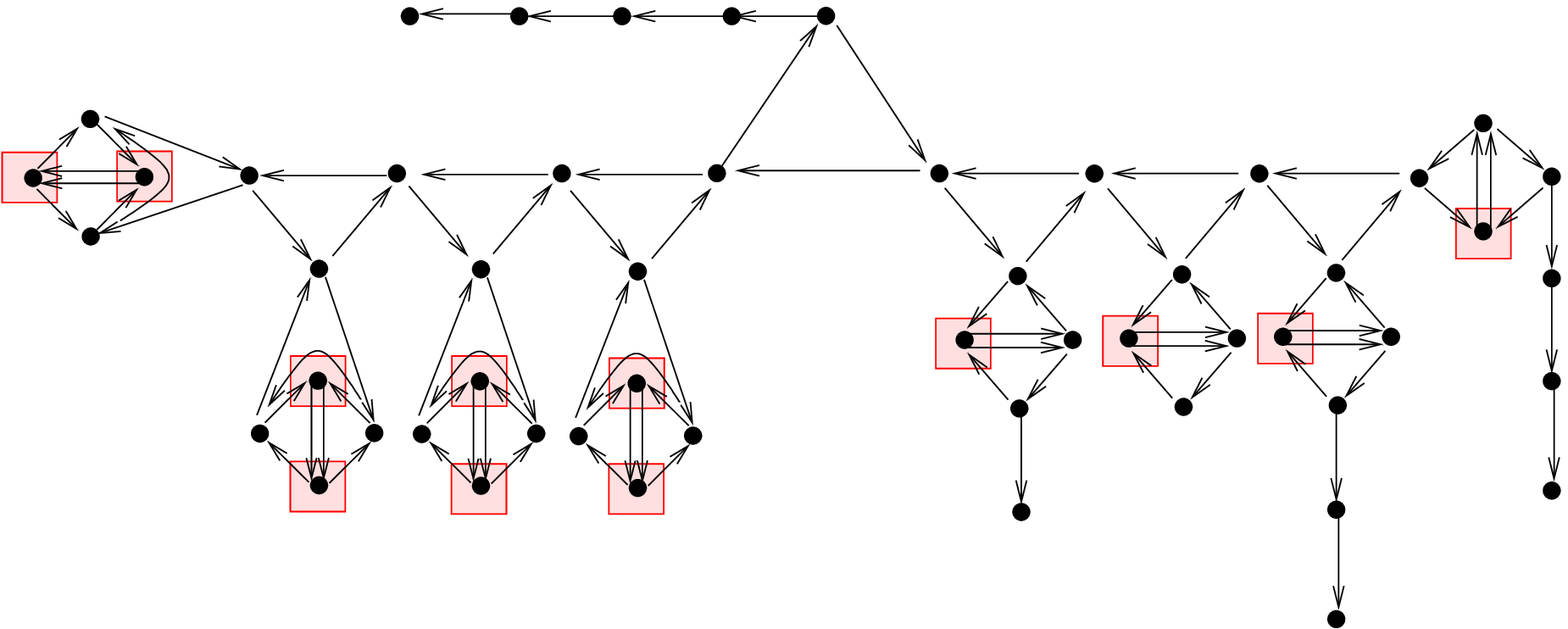,width=0.95\linewidth}
\put(-41,93){\color{blue} $v_k\!\!+\!\!\delta_k$}
\put(-27,147){\color{blue} $v_k$}
\put(-410,149){\color{blue} $v_i$}
\put(-410,98){\color{blue} $v_j$}
\put(-463,110){\color{blue} $v_i\!\!+\!\!v_j\!\!+\!\!\delta_{ij}^1$}
\put(-388,146){\color{blue} $v_i\!\!+\!\!v_j\!\!+\!\!\delta_{ij}^2$}
\put(-392,133){\color{blue} \large / }
\caption{Constructing the quasi-Cartan companion. }
\label{tr-q-vect}
\end{center}
\end{figure}


\subsection{Homomorphism of groups}

Given an unpunctured surface $S$, we triangulate it as in Section~\ref{triang-sp} and obtain a quiver $Q$. Then
we can construct two groups  from the same quiver $Q$:
the group $G=G(Q)$, generated by involutions $s_i$ with relations as in Section~\ref{group}, and the extended affine Weyl group  
$W=W({\bf u},Q)$ generated by reflections $r_i=r_{u_i}$.

\begin{theorem}
\label{homo}
The mapping  $s_i\mapsto r_i$  of the generators extends to a surjective homomorphism
 $\varphi: G\to W$.

\end{theorem}

\begin{proof}
As the generators of $G$ and $W$ are in bijection, we only need to show that the map $f$ takes each defining relation of $G$ to a relation which holds in the reflection group  $W$.
This is clearly the case for relations of types (R1) and (R2) by the construction of $W$. Relations (R4) follow from~\cite[Theorem 1]{Sa3}. 

For the remaining relations (R3) and (R5), the verification is straightforward: we write the matrices of the reflections $r_i$ explicitly and check the corresponding relations. We present the corresponding matrices in the appendix.

\end{proof}

\begin{conjecture}
\label{iso}
The map $\varphi$ in Theorem~\ref{homo} is an isomorphism, i.e. $G(Q)\cong W({\bf u},Q)$. 

\end{conjecture}

\begin{remark}
  \label{aff}
  Conjecture~\ref{iso} holds for surfaces of genus $0$ with at most two boundary components: in this case both groups are either finite or affine Weyl groups~\cite{BM,FeTu}.  

\end{remark}

\begin{remark}
  \label{iso-el}
  Conjecture~\ref{iso} also holds for exceptional mutation-finite quivers $E_6^{(1,1)}$, $E_7^{(1,1)}$ ans $E_8^{(1,1)}$ in the following sense. The construction of a group in Section~\ref{group-constr} can also be applied to the mutation classes of the quivers listed above, see~\cite{FeTu}. The presentations of extended affine Weyl groups for elliptic root systems of these types are given in~\cite{Sa3}. Comparing the presentations, we see that the groups are isomorphic. 

\end{remark}

We will now state our main result: the group $W$ is invariant under mutations, i.e. one can apply mutations to the quasi-Cartan companions.

\begin{theorem}
\label{c-ind}
For any mutation sequence $\mu$, the vectors $\mu(\bf u)$ provide a positive semi-definite quasi-Cartan companion for $\mu(Q)$.   

\end{theorem}

Our proof of Theorem~\ref{c-ind} is based on the notion of an {\em admissible} quasi-Cartan companion introduced by Seven~\cite{Se}, we prove the theorem in Section~\ref{adm-sec}. At the same time, Theorem~\ref{c-ind} can be also considered as a corollary of Conjecture~\ref{iso} (if it holds).

\begin{prop}
\label{ind}
Conjecture~\ref{iso} implies Theorem~\ref{c-ind}.

\end{prop}

\begin{proof}
We will proceed inductively to prove the following claim:\\

\noindent
{\bf Claim.}
{\it
Suppose there is a quiver $Q$ and a vector system ${\bf u} =\{u_i \}$ satisfying the following conditions:
\begin{itemize}
\item[-] $\bf u$ is a companion basis for $Q$;
\item[-] the reflections $\{r_i\}$ (where $r_i=r_{u_i}$) generate an extended affine Weyl group isomorphic to $G=G(Q)$ via the mapping $s_i\mapsto r_{u_i}$. 
\item[]
Then for any mutation $\mu_k$ the quiver $\mu_k(Q)$ and the vector system $\mu_k(\bf u)$ satisfy the same conditions. 
\end{itemize}
}

It is sufficient to show this inductive statement, as we can start with the quiver $Q$ and the corresponding vectors $\{u_i\}$ constructed in Section~\ref{triang-sp} and hence satisfying the assumption.

Let $Q'=\mu_k(Q)$. The group $G(Q')$ coincides with the group $G$ with generators given by 
$$t_i=\begin{cases}
s_ks_is_k & \text{ if $Q$ contains an arrow from $i$ to $k$ },\\
s_i &\text{ otherwise}.
\end{cases}
$$
Applying $\varphi$ (which is an isomorphism by the assumption of the proposition) to these new generators, we obtain the set of reflections in $W$ given by
$$
r_i'=\begin{cases}
r_kr_ir_k & \text{ if $Q$ contains an arrow from $i$ to $k$ },\\
r_i &\text{ otherwise}.
\end{cases}
$$
Notice that each of these new reflections $r_i'$ is a reflection with respect to a new vector $u_i'$ obtained from $u_i$ by a reflection with respect to $u_k$ (see Fig.~\ref{fig-ind}). According to Remark~\ref{geometric realisation}, this is exactly the action of the mutation $\mu_k$ on the vectors $\{u_i\}$.
In particular, the order of the element $t_it_j$ in $G$ coincides with the order of the element $r_i'r_j'$ in $W$,
and hence the vectors $\{ r_i'\}$ form a quasi-Cartan companion for $Q'$.

\end{proof}

\begin{figure}[!h]
\begin{center}
\epsfig{file=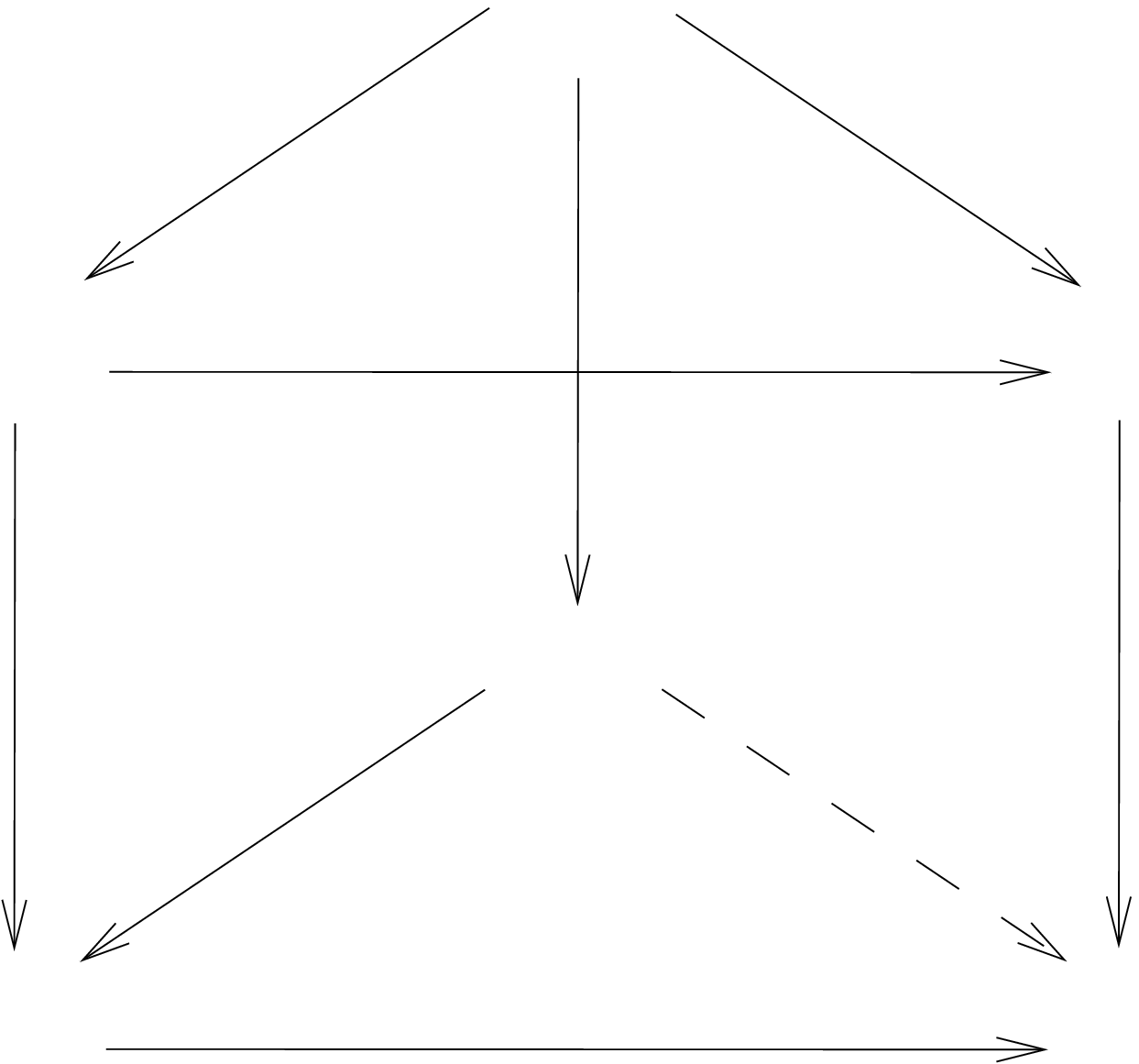,width=0.3\linewidth}
\put(-72,125){$Q$}
\put(-155,82){$G(Q)$}
\put(-240,82){$(s_1,\dots,s_n) =$}
\put(-5,82){$W$}
\put(20,82){$=(r_{u_1},\dots,r_{u_n})$}
\put(-74,43){$Q'$}
\put(-157,0){$G(Q')$}
\put(-240,0){$(t_1,\dots,t_n) =$}
\put(-5,0){$W'$}
\put(20,0){$= (r_{u_1'},\dots,r_{u_n'})$}
\put(-40,88){$\varphi$}
\put(-40,7){$\varphi$}
\put(-80,100){$\mu_k$}
\put(-148,55){$\mu_k$}
\caption{To Proposition~\ref{ind}. }
\label{fig-ind}
\end{center}
\end{figure}

\section{Admissible quasi-Cartan companions}
\label{adm-sec}

\subsection{Proof of Theorem~\ref{c-ind}}

In~\cite{Se}, Seven generalized the notion of a $k$-compatible quasi-Cartan companion of a quiver $Q$.

\begin{definition}[\cite{Se}]
A quasi-Cartan companion $A$ is {\it admissible } if the following holds: for every chordless cycle $i_1,\dots,i_k$ the cyclic product of the elements $-a_{i,i+1}$ is negative if the cycle is oriented, and positive otherwise. 
  \end{definition}

  One can easily see that restricting the admissibility condition to $3$-cycles leads to the definition of a fully compatible companion, and thus an admissible companion is always fully compatible. However, the converse may not be true.

Note that the quasi-Cartan companion constructed in Section~\ref{triang-sp} is fully compatible, positive semi-definite, and does not contain any cycles of length more than $3$, and thus is admissible. Therefore, to prove Theorem~\ref{c-ind}, it is sufficient to prove the following statement.

  \begin{prop}
    \label{adm}
    Let the quiver $Q_0$ and a vector system ${\bf u} =\{u_i \}$ be those constructed in Section~\ref{triang-sp}, and assume that $\mu$ is a sequence of mutations such that $\mu({\bf u})$ provides an admissible quasi-Cartan companion of $\mu(Q_0)$.
 Then for any mutation $\mu_k$ the vector system $\mu_k(\mu(\bf u))$ provides an admissible quasi-Cartan companion of $\mu_k\mu(Q_0)$. 

\end{prop}

{  
\begin{definition}
Given a quiver $Q$ and its quasi-Cartan companion $A$, let us call $A$ a {\em symmetric twin} of $Q$ if every mutation sequence $\mu$ takes $A$ to a quasi-Cartan companion of $\mu(Q)$. 
\end{definition}

Then Proposition~\ref{adm} implies the following corollary.

  \begin{cor}
\label{all-un}
Every quiver constructed from a triangulation of an unpunctured surface has a symmetric twin. The twin is unique up to simultaneous sign changes of rows and columns. 

\end{cor}

\begin{proof}
By Proposition~\ref{adm}, the Gram matrix of the vector system ${\bf u}$ is a symmetric twin of $Q_0$, and thus every quiver mutation-equivalent to it also has a symmetric twin. To prove the uniqueness, notice that a symmetric twin must be admissible. Indeed, if a quasi-Cartan companion $A$ of $Q$ is not admissible, then there exist a cycle $C$ in $Q$ on which the admissibility condition fails. The cycle $C$ itself is a quiver of type $D_k$ or $\t A_{m,n}$. It is now easy to check that the restriction of $A$ onto $C$ is not a symmetric twin of $C$, which implies that $A$ is not a twin of $Q$. Finally, it is proved in~\cite{Se} that an admissible quasi-Cartan companion to a quiver, if exists, is unique up to simultaneous change of sign of rows and columns, which completes the proof.     
  
  \end{proof}

}

\begin{remark}
\label{all-ex}
  Corollary~\ref{all-un} also holds for all exceptional finite mutation classes except for $X_6$ and $X_7$. Indeed, for quivers of finite and affine types this follows from Remark~\ref{aff} (or directly from~\cite{S2}), and for elliptic quivers $E_6^{(1,1)}$, $E_7^{(1,1)}$, $E_8^{(1,1)}$ this follows from Remark~\ref{iso-el}.

  \end{remark}

Let us now prove Proposition~\ref{adm}. First, we use the properties of admissible companions to prove the proposition for a very restricted set of surfaces.

\begin{lemma}
\label{ell-inv}
Proposition~\ref{adm} holds for quivers $Q_0$ constructed from a surface $S$ that is either a genus $0$ surface with three boundary components, or a genus $1$ surface with one boundary component. 
  
\end{lemma}

\begin{proof}
We will proceed inductively: suppose that a quiver $Q=\mu(Q_0)$ for some sequence of mutations $\mu$, and $A=\mu(A^0)$ is its admissible quasi-Cartan companion, where $A^0$ is the admissible companion of $Q_0$ constructed as a Gram matrix of vectors ${\bf u}$. Choose any vertex $k$.
  
    To prove that the mutation $\mu_k(A)$ is admissible, we need to show that the admissibility condition holds for every chordless cycle $Q'_c$ in the quiver $Q'=\mu_k(Q)$. As $Q'$ is a quiver constructed from an unpunctured surface,   $Q'_c$ can be either oriented of length $3$, or non-oriented.

    Observe that $Q'_c$ is of affine or finite type. In particular, if the vertex $k$ belongs to $Q'_c$, then the quiver $\mu_k(Q'_c)$ is also of affine or finite type. Since  $\mu_k(Q'_c)$ is a subquiver of $Q$, the restriction of $A$ to it is also admissible, and thus its mutation restricted to $Q'_c$ is admissible by~\cite[Corollary 1.8]{S2}.
    
    Therefore, we can now assume that the vertex $k$ does not belong to  $Q'_c$, denote $\t Q'_c= Q'_c\cup\{k\}$. 

    Suppose first that $Q'_c$ contains exactly three vertices, and consider a subquiver $\mu_k(\t Q'_c)$ of $Q$. This subquiver contains four vertices and is mutation-finite. It is easy to see that this implies that either it is of finite or affine type, or it is a quiver shown in Fig.~\ref{dread} (it represents a triangulation of a genus $1$ surface with a unique marked point on its boundary). In the former case the restriction of $\mu_k(A)$ to $\t Q'_c$ is admissible by~\cite{Se}, and in the latter case we just need to check that every mutation of the quiver in Fig.~\ref{dread} leads to an admissible companion, which is a short and straightforward calculation (note that an admissible companion is unique up to simultaneous change of signs of rows and columns~\cite[Theorem 2.11]{Se}, so we need to choose one admissible companion, e.g. the one shown in Fig.~\ref{dread}, and perform four mutations).   

    \begin{figure}
      \begin{center}
\epsfig{file=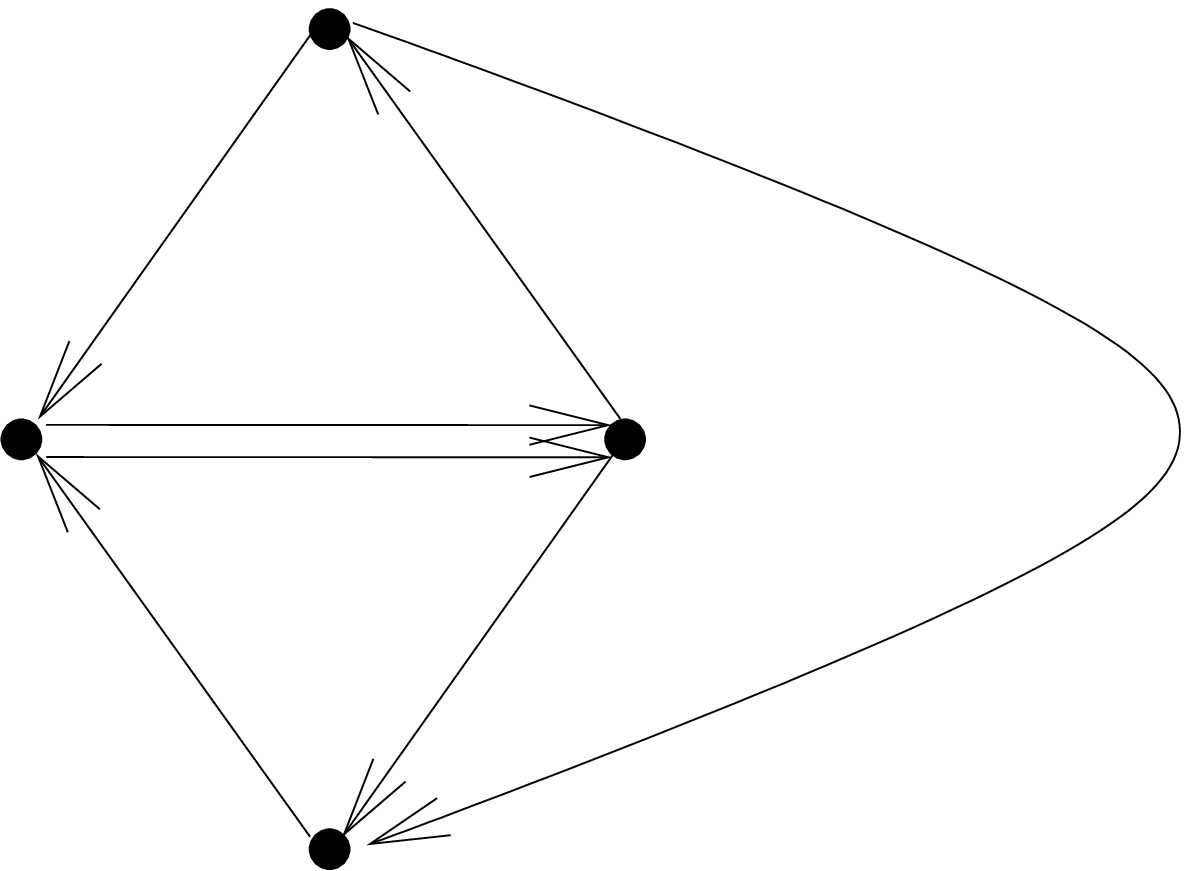,width=0.24\linewidth}
\put(-97,85){\small $e_1-e_2$}
\put(-150,38){\small $e_1-e_3$}
\put(-45,38){\small $e_1-e_3$}
\put(-97,-10){\small $e_2-e_3$}
\end{center}
      \caption{The unique quiver corresponding to triangulations of a torus with one boundary component, and its admissible quasi-Cartan companion}
\label{dread}
    \end{figure}

    Now suppose that  $Q'_c$ contains more than three vertices, in particular, it is non-oriented. By Prop.~\ref{non-or}, any vertex is connected to $Q'_C$ by an even number of arrows. Combining this with the fact that a valence of a vertex in a quiver originating from an unpunctured triangulated surface does not exceed $4$~\cite{FST}, we see that $k$ is connected to $Q'_c$ either by $2$ or by $4$ arrows.

First, suppose that $k$ is connected to $Q'_c$ by $2$ arrows. This cannot be a double arrow, otherwise  $\t Q'_c$ would be mutation-infinite.  If $k$ is connected to non-neighboring vertices of $Q'_c$, then $\t Q'_c$ contains $3$ cycles of length at least $4$ (and thus non-oriented). This contradicts~\cite[Proposition~2.1]{Se1}, according to which a mutation-finite quiver with at least two non-oriented cycles must contain an oriented cycle. Thus, $k$ is connected to two neighboring vertices of $Q'_c$ and forms with them an oriented cycle of length $3$. Then $\mu_k(\t Q'_c)$ is an affine subquiver of $Q$, and thus the restriction of $\mu_k(A)$ to $\t Q'_c$ is admissible by~\cite{Se}.

   \begin{figure}[!h]
      \begin{center}
        \epsfig{file=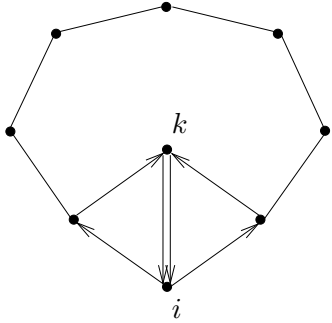,width=0.27\linewidth}
\put(-60,60){\small $k$}
\put(-60,-10){\small $i$}
      \end{center}
      \caption{A quiver $\t Q'_c$ for $k$ being incident to a double arrow. ``Non-oriented'' arrows can be oriented in any way}
\label{wheel1}
    \end{figure}

Suppose now that $k$ is connected to $Q'_c$ by $4$ arrows. If $k$ is connected to some vertex (say, $i$) by a double arrow, then $k$ and $i$ must form oriented cycles with both neighbors of $i$ in $Q'_c$, otherwise the subquiver formed by four vertices $k$, $i$ and two neighbors of $i$ is mutation-infinite. Therefore, $\t Q'_c$ has the form shown in Fig.~\ref{wheel1}. A short  explicit calculation shows that if the restriction of $A$ to $\mu_k(\t Q'_c)$ is admissible, then the restriction of $\mu_k(A)$ to $\t Q'_c$ is also admissible.

So, we can now assume that $k$ is connected to $4$ distinct vertices of  $Q'_c$.   Let us look at possible structure of the quiver  $\t Q'_c$.

By~\cite[Proposition~2.1]{Se1} mentioned above, $k$ must belong to at least one oriented cycle (which is of length $3$ since there are no punctures). Suppose first that $k$ belongs to one oriented cycle only. Then  $\t Q'_c$ contains $3$ non-oriented cycles, see Fig.~\ref{1cycle}. Removing any of the two vertices  of the oriented triangle belonging to $\t Q'_c$ we obtain a subquiver with at least two non-oriented cycles and no oriented cycles. By~\cite[Proposition~2.1]{Se1}, this subquiver is mutation-infinite.

 \begin{figure}[!h]
      \begin{center}
        \epsfig{file=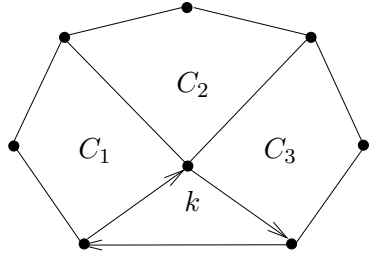,width=0.30\linewidth}
\put(-70,15){\small $k$}
\put(-110,35){\small $C_1$}
\put(-73,60){\small $C_2$}
\put(-40,35){\small $C_3$}
      \end{center}
      \caption{A quiver $\t Q'_c$ for $k$ belonging to exactly one oriented cycle. Cycles $C_i$ are non-oriented and can be of any lengths}
\label{1cycle}
    \end{figure}

Therefore, we can now assume that $k$ belongs to $2$ oriented cycles. If these cycles have a common arrow, then they form a subquiver on $4$ vertices composed of two oriented triangles sharing an edge, which corresponds to a self-folded triangle and thus can never show up in a quiver of unpunctured surface (see~\cite{Gu1,Gu2}). Thus, $\t Q'_c$ consists of two oriented cycles and two non-oriented cycles ``between'' them, see Fig.~\ref{2cycles}. Mutating this quiver in $k$, we obtain one of the two quivers shown in Fig.~\ref{2cycles}. Now an easy computation shows that an admissible companion to the latter mutates to an admissible companion to the former, which completes the proof.

  \begin{figure}[!h]
      \begin{center}
        \epsfig{file=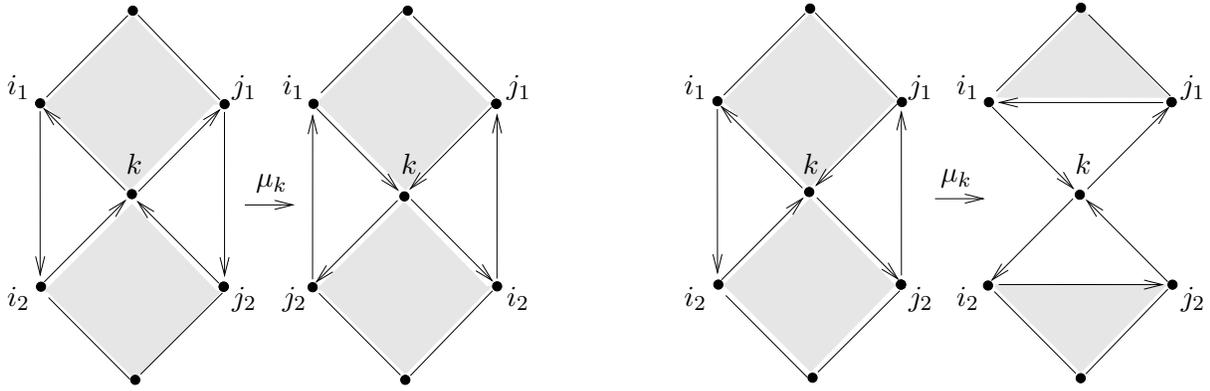,width=0.95\linewidth}
\put(-398,80){\small $k$}
\put(-443,110){\small $i_1$}
\put(-443,30){\small $i_2$}
\put(-358,110){\small $j_1$}
\put(-358,30){\small $j_2$}
\put(-349,75){\small $\mu_k$}
\put(-294,80){\small $k$}
\put(-339,110){\small $i_1$}
\put(-339,30){\small $j_2$}
\put(-254,110){\small $j_1$}
\put(-254,30){\small $i_2$}
\put(-142,80){\small $k$}
\put(-187,110){\small $i_1$}
\put(-187,30){\small $i_2$}
\put(-102,110){\small $j_1$}
\put(-102,30){\small $j_2$}
\put(-90,77){\small $\mu_k$}
\put(-39,80){\small $k$}
\put(-84,110){\small $i_1$}
\put(-84,30){\small $i_2$}
\put(1,110){\small $j_1$}
\put(1,30){\small $j_2$}

      \end{center}
      \caption{Quivers $\t Q'_c$ for $k$ belonging to two oriented cycles and their mutations. Shaded cycles  are non-oriented and can be of any lengths}
\label{2cycles}
    \end{figure}

    \end{proof}

We can now complete the proof of Prop.~\ref{adm} (and thus of Theorem~\ref{c-ind}).

 \begin{proof}[Proof of Proposition~\ref{adm}]   
  
   Let $Q$ be a quiver constructed from a triangulation of an unpunctured surface, and let $A$ be its admissible quasi-Cartan companion. We need to prove that the companion $\mu_k(A)$ of the quiver $Q'=\mu_k(Q)$ is also admissible. As in the proof of Lemma~\ref{ell-inv}, we need  to prove that the admissibility condition holds for every chordless cycle $Q'_c$ in $Q'$.

  Since $Q'_c$ is a cycle, it is a quiver of affine type $\t A_{m-1}$ or finite type $A_m$ (type $D_m$ for $m>3$ is excluded as there are no punctures), where $m$ is the number of vertices in  $Q'_c$. In the former case it corresponds to a triangulated annulus, and in the latter case to a triangulated polygon.  We now consider a subquiver $\t Q'_c= Q'_c\cup\{k\}$ of $Q'$ and determine how can the corresponding to $\t Q'_c$ surface look like.

  Observe that $Q'_c$ is a subquiver of $\t Q'_c$ obtained by removing one vertex. In the language of surfaces, the operation of removing one vertex from a quiver is equivalent to cutting the surface along the corresponding edge of a triangulation. Thus, the surface for $\t Q'_c$ can be obtained from the surface for $Q'_c$ by gluing two segments of the boundary (without creating any punctures), or by attaching to the surface for $Q'_c$ a single triangle along one boundary edge.

  By gluing two segments of boundary of a polygon or by attaching a triangle to it we can obtain either an annulus or a polygon again (closed sphere is excluded), so $\t Q'_c$ is again of affine or finite type. By gluing two segments of boundary of an annulus we can obtain either a $3$-holed sphere (if the segments belong to the same boundary component) or a torus with one boundary component (if the segments belong to distinct components). Similarly, by attaching a triangle to an annulus along one edge we can obtain an annulus only. Thus $\t Q'_c$ is either of affine type, or of one of the types covered by Lemma~\ref{ell-inv}. 

  We now observe that the restriction $A_c$ of $A$ to the subquiver $\mu_k(\t Q'_c)$ of $Q$ is an admissible quasi-Cartan companion. Since the quiver $\mu_k(\t Q'_c)$ is either of finite type, or of affine type, or of one of the types covered by Lemma~\ref{ell-inv}, we deduce that $\mu_k(A_c)$ is also an admissible companion of  $\t Q'_c$ either by~\cite{Se} or by Lemma~\ref{ell-inv}, which completes the proof.

\end{proof}

  \subsection{Constructing an admissible quasi-Cartan companion by a triangulation}
  \label{qCC-c}
  Given an arbitrary triangulation $T$ of an unpunctured surface (or, equivalently, its quiver $Q$), Proposition~\ref{adm} provides a way to construct an admissible quasi-Cartan companion by $T$ without making use of any mutations.

  Let $S$ be a surface of genus $g$ with $b$ boundary components, and let $T$ be any its triangulation. As before, we denote by $n$ the number of interior edges of $T$ (it depends on the number of marked points), we may assume $n\ge 3$.  Cut along some edges of $T$ to obtain a topological disk. Euler characteristics considerations imply that the number of cuts is equal to $2g+b-1$.

  As in Section~\ref{semi-def}, index all the triangles of $T$ (there are $n-2g-b+2$ of them), and consider a Euclidean vector space of dimension $n-2g-b+2$ with orthonormal basis $e_i$. Assign to every edge of $T$ on the disk the vector $e_i+e_j$ if the edge belongs to triangles $i$ and $j$.

  Now, consider any other edge of $T$, assume it belongs to triangles $i$ and $j$. If we glue the disk along this edge only, we obtain an annulus. Therefore, this edge belongs in $Q$ either to an oriented triangle of length $3$ (if triangles $i$ and $j$ have a common edge in the disk), or to non-oriented cycle (otherwise). In the former case the admissibility condition implies that the inner product of the corresponding vector with the two other in the cycle must be positive, and thus we can take a vector $e_i+e_j$ without loss of generality. In the latter case the admissibility condition implies that we must take either $e_i+e_j$ (if the length of the cycle is even) or $e_i-e_j$  (if the length of the cycle is odd).

  Therefore, we have assigned a vector to every edge of a triangulation $T$ and obtained a semi-positive quasi-Cartan companion $A$.

  \begin{cor}
    \label{alg}
    $A$ is an admissible quasi-Cartan companion of $Q$.

    \end{cor}

    \begin{proof}
 The moduli of inner products of constructed vectors match the moduli of the entries of the matrix $B$, and the choice of signs in the construction was unique (up to the change of the sign of any of the vectors) to satisfy the admissibility condition. Therefore, if $A$ is not admissible, then $Q$ has no admissible companions. However, Proposition~\ref{adm} implies that there exists an admissible quasi-Cartan companion of $Q$, which completes the proof.  

      \end{proof}

      \begin{remark}
        \label{a2}
Note that although the vectors we constructed have the form $e_i\pm e_j$, all of them belong to a finite root system of type $A_{n-2g-b+1}$ spanned by the vectors corresponding to the interior edges of the triangulation of the disk.
        \end{remark}

\subsection{Admissible companions and reflection groups}

Finally, we would like to mention a geometric interpretation of the admissibility condition. 

\begin{prop}
  \label{lin-ind}
  Let $Q$ be a quiver of affine type $\tilde A_{p,q}$, let $A$ be a quasi-Cartan companion with companion basis ${\bf u}=\{u_1,\dots,u_{p+q}\}$. For any subquiver $Q'$ of $Q$ denote by $W(A,Q')$ the group generated by reflections in vectors of $\bf u$ assigned to vertices of $Q'$. Then $A$ is admissible if and only if for any $Q'\subset Q$ of affine or finite type the group $W(A,Q')$ is isomorphic to an affine (respectively, finite) Weyl group. 

\end{prop}

\begin{proof}
  Recall that, according to~\cite{Se}, any quiver of affine type has a unique admissible quasi-Cartan companion up to equivalence (corresponding to changing the signs of some of the vectors in the companion basis), and any mutation of an admissible companion of an affine quiver results in an admissible companion again. Thus, if $A$ is admissible, then it can be obtained by mutations from an admissible companion of an affine Dynkin diagram for which the isomorphism between the corresponding reflection group $W(A,Q)$ and $\tilde A_{p+q-1}$ is obvious. We then can use the claim from the proof of Proposition~\ref{ind} to show that $W(A,Q)$ is also isomorphic to $\tilde A_{p+q-1}$. Restricting this reasoning to any affine subquiver, we obtain ``only if'' statement.

  Conversely, assume that $A$ is not admissible. This implies that there is a chordless cycle for which the admissibility condition does not hold. There are three types of chordless cycles in quivers of type $\tilde A$: oriented cycles of length $3$ with weights $(1,1,1)$ or $(1,1,2)$, and non-oriented cycles with all weights equal to one. If the admissibility condition is broken for a non-oriented cycle $Q'$, then $W(A,Q')$ is a finite Weyl group of type $D$; if it  is broken for a cycle of type $(1,1,2)$, then the corresponding reflection group is isomorphic to a group generated by reflections in sides of a hyperbolic triangle with angles $(\pi/3,\pi/3,0)$. As both of these types of cycles are affine subquivers themselves, we see that in both cases we have an affine subquiver $Q'$ with the group $W(A,Q')$ not being an affine Weyl group. Finally, if the admissibility condition is broken for an oriented $3$-cycle with weights $(1,1,1)$, then the corresponding reflection group is isomorphic to an affine Weyl group $\tilde A_2$. As the cycle itself is a quiver of finite type $A_3$, we again come to a contradiction.

\end{proof} 

Combining  Propositions~\ref{lin-ind} and~\ref{adm} we obtain the following result.

\begin{cor}
\label{aff-adm}
Let the quiver $Q_0$ and a vector system ${\bf u} =\{u_i \}$ be those constructed in Section~\ref{triang-sp}. 
Then for any sequence of mutations $\mu$ and any subquiver $Q'\subset \mu(Q_0)$ of affine or finite type the group $W(\mu_k(A),Q')$ is isomorphic to an affine (respectively, finite) Weyl group.
    
\end{cor}


\section*{Appendix: Calculations in the proof of Theorem~\ref{homo}}

We present below the calculations for the relations of types (R3) and (R5), which are required to prove Theorem~\ref{c-ind}.

We index the vertices as in Fig.~\ref{rel}(e), and assign the following vectors to the vertices:
$$u_1=e_2-e_3+\delta_1,\quad u_2=e_1-e_3,\quad u_3=e_2-e_3+\delta_2,\quad u_4=e_1-e_2,\quad u_5=e_1-e_4.$$
Here we assume that the vector space $V_+$ with basis $\{v_1,\dots,v_{n_+}\}$ is embedded in the vector space $V'_+$ of dimension $n_++1$ with orthonormal basis $\{e_i\}$ (and thus the whole space $V\oplus V_0^*$ is embedded in the space with basis $\{e_1,\dots,e_{n_++1},\delta_1,\dots,\delta_{n_0},\delta_1^*,\dots,\delta_{n_0}^*\}$), and all roots lying in $V_+$ are of the form $e_i-e_j$ (see Section~\ref{eawg} for the notation and further details). The action of the group $W$ on $V=V_+\oplus V_0$ can be naturally extended to an action on  $V'=V'_+\oplus V_0$. 

It is easy to see that reflections $s_1,\dots,s_5$ in the vectors $u_1,\dots,u_5$ may only act non-trivially on the $8$-dimensional subspace of $V'$ spanned by vectors $\{e_1,e_2,e_3,e_4,\delta_1,\delta_2,\delta_1^*,\delta_2^*\}$. Therefore, to verify the relations (R3) and (R5) we may write down the matrices of the reflections in the basis above and calculate the required products. The matrices of the reflections in this basis have the following form:

$$
s_1 = \begin{pmatrix}
  1& 0& 0& 0& 0& 0& 0& 0\\
  0& 0& 1& 0& 0& 0& -1& 0\\
  0& 1& 0& 0& 0& 0& 1& 0\\
  0& 0& 0& 1& 0& 0& 0& 0\\
  0& -1& 1& 0& 1& 0& -1& 0\\
  0& 0& 0& 0& 0& 1& 0& 0\\
  0& 0& 0& 0& 0& 0& 1& 0\\
  0& 0& 0& 0& 0& 0& 0& 1
\end{pmatrix}\qquad\qquad
s_2 = \begin{pmatrix}
  0& 0& 1& 0& 0& 0& 0& 0\\
  0& 1& 0& 0& 0& 0& 0& 0\\
  1& 0& 0& 0& 0& 0& 0& 0\\
  0& 0& 0& 1& 0& 0& 0& 0\\
  0& 0& 0& 0& 1& 0& 0& 0\\
  0& 0& 0& 0& 0& 1& 0& 0\\
  0& 0& 0& 0& 0& 0& 1& 0\\
  0& 0& 0& 0& 0& 0& 0& 1
\end{pmatrix}
$$
$$
s_3 = \begin{pmatrix}
  1& 0& 0& 0& 0& 0& 0& 0\\
  0& 0& 1& 0& 0& 0& 0& -1\\
  0& 1& 0& 0& 0& 0& 0& 1\\
  0& 0& 0& 1& 0& 0& 0& 0\\
  0& 0& 0& 0& 1& 0& 0& 0\\
  0& -1& 1& 0& 0& 1& 0& -1\\
  0& 0& 0& 0& 0& 0& 1& 0\\
  0& 0& 0& 0& 0& 0& 0& 1 
\end{pmatrix}
\qquad\qquad
s_4 = \begin{pmatrix}
  0& 1& 0& 0& 0& 0& 0& 0\\
  1& 0& 0& 0& 0& 0& 0& 0\\
  0& 0& 1& 0& 0& 0& 0& 0\\
  0& 0& 0& 1& 0& 0& 0& 0\\
  0& 0& 0& 0& 1& 0& 0& 0\\
  0& 0& 0& 0& 0& 1& 0& 0\\
  0& 0& 0& 0& 0& 0& 1& 0\\
  0& 0& 0& 0& 0& 0& 0& 1
\end{pmatrix}
$$
$$
s_5 = \begin{pmatrix}
  0& 0& 0& 1& 0& 0& 0& 0\\
  0& 1& 0& 0& 0& 0& 0& 0\\
  0& 0& 1& 0& 0& 0& 0& 0\\
  1& 0& 0& 0& 0& 0& 0& 0\\
  0& 0& 0& 0& 1& 0& 0& 0\\
  0& 0& 0& 0& 0& 1& 0& 0\\
  0& 0& 0& 0& 0& 0& 1& 0\\
  0& 0& 0& 0& 0& 0& 0& 1
\end{pmatrix}
$$

 Now a direct computation shows that both  $(s_1\ s_2s_3s_4s_3s_2)^3$ and $(s_1\ s_2s_3s_4s_5s_4s_3s_2)^2$ are identity matrices, which verifies relations (R5). Finally, to verify the relation (R3), we compute $(s_2\ s_4s_5s_4)^2$ which also turns out to be an identity matrix, as required.  

Note that the relations (R1) and (R2) can also be checked straightforwardly using the matrices provided above, and to verify (R4) we need to assign a different vector to the vertex $v_4$, say, $e_2-e_4$ (cf. Fig.~\ref{rel}(c)).

\end{document}